\documentclass[9pt]{amsart}
\usepackage[]{amsmath, amsthm, amsfonts,  scalerel}
\usepackage{graphicx}
\usepackage[utf8]{inputenc}
\usepackage[english]{babel}
\usepackage[all]{xy}
 \usepackage{xcolor}
 \usepackage{tikz-cd}
\usepackage{adjustbox}
\usepackage{mathtools}
\usepackage{hyperref}
\usetikzlibrary{datavisualization}
\usetikzlibrary{datavisualization.formats.functions}
\usepackage{pgfplots}
\newcommand{\boxtimes}{\mathbin{\scalerel*{\tikz{\draw[line width=1.1pt](0,0)rectangle(1,1)--(0,0)(1,0)--(0,1);}}{\otimes}}}
\usepackage[
backend=biber,
style=alphabetic
]{biblatex}

\usepackage[left=3cm,right=3cm,top=3cm,bottom=3cm]{geometry}
\usepackage{cleveref}
\usepackage[left=3cm,right=3cm,top=3cm,bottom=3cm]{geometry}
\usepackage{cleveref}

\newcommand {\N} {{\mathbb N}}
\newcommand {\C} {{\mathbb C}}
\newcommand {\fC} {{\mathfrak{C}}}

\newcommand {\Z} {{\mathbb Z}}
\newcommand {\Q} {{\mathbb Q}}
\newcommand {\cF} {{\mathcal F}}

\newcommand {\cY} {{\mathcal Y}}

\newcommand {\cH} {{\mathcal H}}

\newcommand {\cL} {{\mathcal L}}

\newcommand {\cA} {{\mathcal A}}
\newcommand {\cE} {{\mathcal E}}

\newcommand {\cM} {{\mathcal M}}
\newcommand {\cV} {{\mathcal V}}
\newcommand {\cU} {{\mathcal U}}
\newcommand {\bfj} {{\mathbf{j} }}

\newcommand {\cO} {{\mathcal O}}
\newcommand {\D} {\mathbb{D}}

\newcommand{\tX}{\tilde{X}}
\newcommand{\tY}{\tilde{Y}}

\newcommand{\cD}{\mathcal{D}}
\newcommand{\cN}{\mathcal{N}}
\newcommand{\cX}{\mathcal{X}}
\newcommand{\bR}{\textbf{R}}

\newcommand{\Om}{\underline{\Omega}_{X}^{0}}
\newcommand{\OmY}{\underline{\Omega}_{Y}^{0}}
\newcommand{\Omp}{\underline{\Omega}_{X}^{p}}
\newcommand{\OmpY}{\underline{\Omega}_{Y}^{p}}

\DeclareMathOperator{\im}{im}

\DeclareMathOperator{\Spec}{Spec}

\DeclareMathOperator{\lct}{lct}
\DeclareMathOperator{\lcdef}{lcdef}

\newcommand{\bT}{\begin{tikzcd}}
\newcommand{\eT}{\end{tikzcd}}

\newtheorem{thm}[subsection]{Theorem}
\newtheorem{cor}[subsection]{Corollary}
\newtheorem{lemma}[subsection]{Lemma}
\newtheorem{prop}[subsection]{Proposition}
\newtheorem{defn}[subsection]{Definition}
\newtheorem{rmk}[subsection]{Remark}
\newtheorem{ex}[subsection]{Example}

\newtheorem{set}[subsection]{Setting}
\newtheorem{nota}[subsection]{Notation}

\begin{document}
\author{ Scott Hiatt }
\date{\today}
  \address{
 Department of Mathematics\\
  University of Wisconsin-Madison\\
  Madison, WI 53706\\
  U.S.A.}
  \email{shiatt@wisc.edu}

 \title{Cyclic covers via Mixed Hodge modules }

\maketitle

\begin{abstract}
    Suppose we are given an arbitrary line bundle $\cL$ on a complex algebraic variety $X$, not necessarily smooth. For a positive integer $N$, suppose there exists a global section $s \in \Gamma(X, \cL^{N})$ that defines an effective Cartier divisor $D$. If we denote $\pi: Y \rightarrow X$ to be the $N$-fold cyclic covering of the divisor $D$ resulting from the global section $s$, we describe the cyclic covering in terms of Morihiko Saito's theory of mixed Hodge modules.
\end{abstract}


\section*{Introduction}
Suppose that we are given a complex affine variety $X$ with a section $f \in \Gamma(X, \cO_{X}).$ For any $N \in \N$ we can consider the hypersurface $Y=\{t^{N} -f = 0\} \subset X \times \C$. The variety $Y$ is called the $N$-fold cyclic covering of $D = \{f = 0\},$ and there is finite map $\pi:Y \rightarrow X$ with the commutative diagram
$$\bT Y \ar[rr, hook] \ar[dr, "\pi"] & & X \times \C \ar[dl, "\pi_{1}"]\\
& X & \eT$$
where $\pi_{1}:X \times \C \rightarrow X$ is the natural projection map. This construction can be generalized for any complex algebraic variety $X$. Suppose we are given a line bundle $\cL$ on a complex algebraic variety $X$ such that for some positive integer $N$ there exists a global section $s \in \Gamma(X, \cL^{N})$ that defines an effective Cartier divisor $D$. Given an affine covering $\{X_{\alpha}\}_{\alpha \in I}$ of $X$ such that $\cL|_{X_{\alpha}} \cong \cO_{X_{\alpha}}$ for each $\alpha \in I$, let $\pi_{\alpha}: Y_{\alpha} \rightarrow X_{\alpha}$ be the $N$-fold cyclic covering of the section $s|_{X_{\alpha}}:= s_{\alpha} \in \Gamma(X_{\alpha}, \cO_{X_{\alpha}})$, described previously.  We can construct the $N$-fold cyclic covering of $D$, given by $Y= \displaystyle \Spec\bigg(\bigoplus_{i=0}^{N-1}\cL^{-i}\bigg)$, with a surjective map $\pi: Y \rightarrow X$, by ``gluing" the local data from the maps $\pi_{\alpha}: Y_{\alpha} \rightarrow X_{\alpha}$.  Moreover, there exits a divisor $D':= \pi^{*}(D) \subset Y$ such that the induced map $\pi:D'_{red} \rightarrow D_{red}$ is an isomorphism. Also, if we set $U = X \backslash D$ and $U' = Y \backslash D'$, the induced map $\pi: U' \rightarrow U$ is a finite \'etale map. If $X$ and the divisor $D$ are smooth, then $Y$ is smooth. We are interested in the case where $D$ is singular, which makes $Y$ singular. We study the variety $Y$ using the theory of mixed Hodge modules. The theory of mixed Hodge modules allows us to derive the following decomposition theorem for cyclic coverings.  
\begin{thm}\label{CyclicThm}
     Let $\cL$ be an arbitrary line bundle on an algebraic variety $X$ of dimension $n$. Suppose that for a positive integer $N$ there exists a global section $s \in \Gamma(X, \cL^{N})$
    that defines an effective Cartier divisor $D.$ Let $\pi: Y \rightarrow X$ denote the $N$-fold cyclic covering of $D$.  If $j: U= X \backslash D \hookrightarrow X$ is the natural map, then there exists $\cV \in D^{b}MHM(U)$  such that we have the following isomorphism in the derived category of mixed Hodge modules,
    $$\pi_{+}\Q^{H}_{Y}[n] \simeq \Q^{H}_{X}[n] \oplus j_{!}\cV.$$
     For $0 \leq p \leq n=\dim(X)$, if we apply the functor $Gr^{F}_{-p}DR(\bullet):D^{b}MHM(X) \rightarrow D^{b}_{coh}(\cO_{X})$ to the isomorphism above, we obtain the quasi-isomorphism
    $$\pi_{*}\OmpY  \simeq \underline{\Omega}^{p}_{X} \oplus Gr^{F}_{-p}DR( j_{!}\cV)[p-n].$$
    \end{thm}
    After a brief discussion of mixed Hodge modules in Section \ref{S1}, we give a proof of the main theorem in Section \ref{S2}. For the rest of the paper, we focus on describing the complex $Gr^{F}_{-p}DR( j_{!}\cV)$ given in Theorem \ref{CyclicThm}. 
    
    When $X$ is a smooth irreducible variety, we may describe the complex $Gr^{F}_{-p}DR( j_{!}\cV)$ in terms of log resolutions. Take any strong log resolution $\phi:\tX \rightarrow X$ of the pair $(X, D)$. That is, $\phi:\tX \rightarrow X$ is a projective morphism where $\tX$ is smooth, $E= \phi^{*}D_{red}$  is a simple normal crossing divisor, and the induced map $\phi: \tX \backslash E \rightarrow X \backslash D$ is an isomorphism. We now have the commutative diagram
$$ \bT \tY \arrow[d, "\tilde{\pi}"] \arrow[r]& Y  \ar[d, "\pi"] \\
  \tX \arrow[r, "\phi"] & X \eT$$
  where $\tY = (Y \times_{X}\tX)$. Under this construction, $\tY = \displaystyle \Spec \bigg( \bigoplus_{i=0}^{N-1} \phi^{*}\cL^{-i} \bigg)$ is the $N$-fold cyclic covering of  $\phi^{*}D$. Furthermore, if $\varrho: \tX \backslash E \hookrightarrow \tX$ is the natural inclusion, by proper base change, there is an isomorphism
$$\tilde{\pi}_{+}\Q^{H}_{\tY}[n] \cong \Q^{H}_{\tX}[n] \oplus \varrho_{!}\cV,$$
where we identify $\cV \in HM_{\tX \backslash E}(\tX \backslash E, n)$ via the isomorphism $\phi: \tX \backslash E \rightarrow X \backslash D$. Since the map $\phi: \tX \rightarrow X$ is projective, for $0 \leq p \leq n = \dim(X)$, we have identifications
$$ \bR \phi_{*}Gr^{F}_{-p}DR(\varrho_{!}\cV) \simeq Gr^{F}_{-p}DR(\phi_{*}\varrho_{!}\cV) \simeq Gr^{F}_{-p}DR(j_{!}\cV).
$$
So, the problem of describing the complex $Gr^{F}_{-p}DR( j_{!}\cV)$ simplifies to the case when $D$ has simple normal crossing support, which is thoroughly discussed in Section \ref{S3}. After we cover the case of simple normal crossing support, we prove the following theorem, which not only describes the complex $\pi_{*}\OmpY$, but its dual complex $\cD(\pi_{*}\underline{\Omega}^{p}_{Y})= \bR \cH om_{\cO_{X}}(\pi_{*}\underline{\Omega}^{p}_{Y}, \omega_{X}) \simeq \pi_{*}\bR \cH om_{\cO_{Y}}(\underline{\Omega}^{p}_{Y}, \omega_{Y}) = \pi_{*}\cD(\underline{\Omega}^{p}_{Y})$.
\begin{thm}\label{MainThm2}
    Let $X$ be a smooth irreducible complex variety of dimension $n$. Let $\cL$ be an arbitrary line bundle on $X$, and suppose that for a positive integer $N $ there exists a global section $s \in \Gamma(X, \cL^{N})$
    that defines an effective Cartier divisor $D.$ If $Y \rightarrow X$ is this $N$-fold cyclic covering of $D$ and $\phi:\tX \rightarrow X$ is a strong log resolution of the pair $(X,D)$ with $E= \phi^{*}D_{red}$, then for $1 \leq i \leq N-1$ there exists effective divisors $D^{> 0}_{i}$ and $D^{ \geq0}_{i}$ on $\tX$ such that for $0 \leq p \leq n,$
    $$\pi_{*}\OmpY \simeq \Omega^{p}_{X} \oplus \bigoplus^{N-1}_{i=1} \bR \phi_{*}(\Omega^{p}_{\tX}\log(E)\otimes \cO_{\tX}(D^{> 0}_{i}))\otimes \cL^{-i}.$$
    $$\pi_{*}\cD(\underline{\Omega}^{n-p}_{Y}) \simeq \Omega^{p}_{X} \oplus \bigoplus^{N-1}_{i=1} \bR \phi_{*}(\Omega^{p}_{\tX}\log(E)\otimes \cO_{\tX}(D^{ \geq0}_{i}))\otimes \cL^{-i}.$$ 
\end{thm}
If $\phi:\tX \rightarrow X$ is a strong log resolution of the pair $(X,D)$ given in the theorem above, with  $\phi^{*}(D)= \displaystyle \sum a_{j}E_{j},$ where $E_{j}$ are the irreducible components of the divisor $E = \phi^{*}D$, then the divisors $D^{> 0}_{i}$ and $D^{ \geq0}_{i}$ are constructed in the following way, 
$$D^{\geq 0}_{i} = \displaystyle \sum \bigg\lfloor \frac{ia_{j}}{N} \bigg \rfloor E_{j}:= \displaystyle
 \bigg\lfloor \frac{i}{N}\phi^{*}D \bigg\rfloor$$
 
 $$D^{> 0}_{i} = \displaystyle \sum b_{j}E_{j} \quad  \text{where   $b_{j} =$}\begin{cases}
    \displaystyle\bigg\lfloor \frac{ia_{j}}{N} \bigg \rfloor & \text{ if  $\displaystyle\frac{ia_{j}}{N} \notin \Z$}\\ \\
     \displaystyle \frac{ia_{j}}{N} -1 & \text{ if  $\displaystyle \frac{ia_{j}}{N} \in \Z$}
 \end{cases}$$
 
 In particular, if $p = 0$ in the Theorem \ref{MainThm2}, we obtain quasi-isomorphisms
 $$\pi_{*}\OmY \simeq \cO_{X} \oplus \bigoplus^{N-1}_{i=1} \bR \phi_{*}( \cO_{\tX}(D^{> 0}_{i}))\otimes \cL^{-i}$$
    $$\pi_{*}\cD(\underline{\Omega}^{n}_{Y}) \simeq \cO_{X} \oplus \bigoplus^{N-1}_{i=1} \bR \phi_{*}( \cO_{\tX}(D^{ \geq0}_{i}))\otimes \cL^{-i}.$$
If $\varphi: \tY \rightarrow Y$ is a resolution of singularities, then $\cD(\underline{\Omega}^{n}_{Y}) \simeq \bR \varphi_{*}\cO_{\tY}.$ Thus, we have quasi-isomorphisms
$$\pi_{*}\bR \varphi_{*}\cO_{\tY} \simeq \pi_{*}\cD(\underline{\Omega}^{n}_{Y}) \simeq \cO_{X} \oplus \bigoplus^{N-1}_{i=1} \bR \phi_{*}( \cO_{\tX}(D^{ \geq0}_{i}))\otimes \cL^{-i}.$$
 Recall that the variety $Y$ is said to have Du Bois singularities if the canonical map $\cO_{Y} \rightarrow \OmY$ is a quasi-isomorphism, and rational singularities if the canonical map $\cO_{Y} \rightarrow \bR \varphi_{*}\cO_{\tY}$ is a quasi-isomorphism. Because the map $\pi:Y \rightarrow X$ is affine, we can say $Y$ has Du Bois singularities if and only if the map
$$\cO_{X} \oplus \displaystyle   \bigoplus_{i =1}^{N-1} \cL^{-i} \cong \pi_{*}\cO_{X} \rightarrow \pi_{*}\OmY \simeq \cO_{X} \oplus \bigoplus^{N-1}_{i=1} \bR \phi_{*}( \cO_{\tX}(D^{> 0}_{i}))\otimes \cL^{-i}$$
is a quasi-isomorphism. This reduces to $Y$ having Du Bois singularities if and only if the natural map
 $$\cO_{X} \rightarrow \bR \phi_{*}( \cO_{\tX}(D^{> 0}_{i}))  \quad \text{is a quasi-isomorphism for $1 \leq i \leq N-1.$}$$
 Similarly, $Y$ has rational singularities if and only if the natural map
 $$\cO_{X} \rightarrow \bR \phi_{*}( \cO_{\tX}(D^{\geq 0}_{i})) \quad \text{is a quasi-isomorphism for $1 \leq i \leq N-1.$}$$
 The divisors $D^{>0}_{i}$ and $D^{\geq 0}_{i}$ are closely related the log canonical threshold of the divisor $D$. For $\lambda \in \Q$, the \emph{multiplier ideal sheaf} associated to the $\Q$-divisor $\lambda \cdot D$ is defined to be
        $$\mathcal{J}(\lambda \cdot D)= \phi_{*}\cO_{\tX}(K_{\tX/X} - \lfloor \lambda \cdot \phi^{*}(D) \rfloor) \subseteq \cO_{X}.$$
        For a closed point $x \in X$, the \emph{log canonical threshold} of the divisor $D$ at $x$ is defined by
        $$\lct(D,x):= \inf \bigg \{ \lambda\in \Q \hspace{.05in}| \hspace{.05in} \mathcal{J}(X, \lambda \cdot D)_{x} \subseteq m_{x} \bigg\},$$
        where $m_{x} \subset \cO_{X}$ is the maximal ideal sheaf of $x.$ In general, the log canonical threshold of $D$ is defined as
        $$\lct(D):= \inf \bigg \{ \lct(D,x) \hspace{.05in}| \hspace{.05in} x \in X\bigg \}$$
        As an application of Theorem \ref{CyclicThm}, we can prove the following relationship between the log canonical threshold of the divisor $D$ and the singularities of the $N$-fold cyclic covering $Y$.
         \begin{thm}\label{lctThm}
           If $Y \rightarrow X$ is this $N$-fold cyclic covering of $D$ and $\phi:\tX \rightarrow X$ is a strong log resolution of the pair $(X,D)$, 
       $$\text{The natural map $\cO_{X} \rightarrow \bR \phi_{*}( \cO_{\tX}(D^{> 0}_{i}))$  is a quasi-isomorphism} \quad \text{if and only if} \quad \lct(D)\geq \dfrac{i}{N}$$
        $$\text{The natural map $\cO_{X} \rightarrow \bR \phi_{*}( \cO_{\tX}(D^{\geq 0}_{i}))$  is a quasi-isomorphism} \quad \text{if and only if} \quad \lct(D)> \dfrac{i}{N}$$
       \end{thm}
        Assume $D =\{f= 0\}$ is a hypersurface in a smooth affine variety $X$ and let $Y =\{t^{N}-f=0\} \subset X \times \C$ be the $N$-fold cyclic covering of $D$. There exists a polynomial $b(s) \in \C[s]$, and a polynomial $P(s) \in \cD_{X}[s]$ satisfying the relation
        $$P(s)f^{s+1} = b(s) \cdot f^{s}.$$
        The set of all polynomials $b(s)$ forms an ideal in the polynomial ring $\C[s]$, and the monic generator of this ideal, denoted as $b_{f}(s),$ is called the \emph{Bernstein-Sato polynomial} or the $\emph{b-function}$ of $f$. The minimal root of the reduced Bernstein-Sato polynomial $ \tilde{b}_{f}(-s):=\displaystyle \frac{b_{f}(-s)}{s-1}$, which is denoted as $\tilde{\alpha}_{f}$, is called the \emph{minimal exponent of $f$}. The relationship between the minimal exponent of the local defining equation $f$ and the log canonical threshold of $f$ is given by
        $$\lct(f) = \min \bigg \{ \tilde{\alpha}_{f}, 1 \bigg \}.$$ It was shown by Saito \cite{Saito-micr0} that $\tilde{\alpha}_{(t^{N}-f)} = \tilde{\alpha}_{f} + \tilde{\alpha}_{t^{N}} = \tilde{\alpha}_{f} + \dfrac{1}{N}.$ By \cite{Saito-DB}, $Y$ has Du Bois singularities if and only if 
        $$\tilde{\alpha}_{(t^{N}-f)} = \tilde{\alpha}_{f} + \dfrac{1}{N} \geq 1.$$
        By \cite{Saito-Rat}, $Y$ has rational singularities if and only if 
        $$\tilde{\alpha}_{(t^{N}-f)} = \tilde{\alpha}_{f} + \dfrac{1}{N} > 1.$$
        We may recover these results from the theorem above. The variety $Y$ has Du Bois singularities if and only if the natural map $\cO_{X} \rightarrow \bR \phi_{*}( \cO_{\tX}(D^{> 0}_{i}))$  is a quasi-isomorphism for $1 \leq i \leq N-1.$ By the theorem above, $Y$ has Du Bois singularities if and only if 
        $$\tilde{\alpha}_{f} \geq \lct(f) \geq \dfrac{N-1}{N} = 1 - \dfrac{1}{N}.$$
        Similarly, $Y$ has rational singularities if and only if 
        $$\tilde{\alpha}_{f} \geq \lct(f) > \dfrac{N-1}{N} = 1 - \dfrac{1}{N}.$$
        
In Section \ref{S4}, we generalize our approach for singular varieties by using simplicial resolutions. Let $\epsilon: X_{\bullet} \rightarrow X$ be a cubical hyperresolution. That is,  $\epsilon: X_{\bullet} \rightarrow X$ is a simiplical resolution with $X_{i}$ smooth for $i\geq 0$ and $\dim X_{n-i} \leq i$. We may choose such a simplicial resolution by \cite{GNPP}. We set $X_{-1}= X.$ There are proper maps
$$\epsilon_{ij}:X_{i} \rightarrow X_{i-1} \quad \text{for $0 \leq j \leq i$}$$
which satisfy the conditions
$$\epsilon_{ij}\epsilon_{i+1, j+1} = \epsilon_{ij}\epsilon_{i+1,j} \quad \text{for all $j<i.$}$$
By composing any sequence of $(\epsilon_{ij})$, we obtain a well defined morphism
$$\epsilon^{i}:X_{i} \rightarrow X.$$
Now, for each $i$, there is a commutative diagram
$$ \bT Y_{i} \arrow[d, "\pi_{i}"] \arrow[r,"\varepsilon^{i}"]& Y  \ar[d, "\pi"] \\
  X_{i} \arrow[r, "\epsilon^{i}"] & X \eT$$
  where $Y_{i} = (Y \times_{X} X_{i})_{red}$ is the $N$-fold cyclic covering of  $D_{i}=(\epsilon^{i})^{*}D$. Since $X_{i}$ is smooth (not necessarily irreducible), we can apply the results from Section \ref{S3}. For $1 \leq i \leq N-1$ and $0 \leq p \leq n = \dim(X)$, we construct the complex $\underline{\Omega}^{p}_{X}(\log D) (D_{i}^{>0})$, which can be thought of as the ``twisted" $p^{th}$-graded piece of the Du Bois complex of the pair $(X, D)$. With this construction, we obtain our general decomposition for the complex $\pi_{*}\OmpY$.
  \begin{thm}
       Let $\cL$ be an arbitrary line bundle on an algebraic variety $X$ of dimension $n$. Suppose that for a positive integer $N$ there exists a global section $s \in \Gamma(X, \cL^{N})$
    that defines an effective Cartier divisor $D.$ Let $\pi: Y \rightarrow X$ denote the $N$-fold cyclic covering of $D$. For $0 \leq p \leq n$, 
    $$\pi_{*}\underline{\Omega}^{p}_{Y} \simeq \Omp \oplus \bigoplus_{i=1}^{N-1}(\underline{\Omega}^{p}_{X}(\log D) (D_{i}^{>0})) \otimes \cL^{-i}.$$
  \end{thm}
A precise definition of the complex $\underline{\Omega}^{p}_{X}(\log D) (D_{i}^{>0})$ is given in Section \ref{S4}.  If we choose $D$ sufficiently general, for example, when $D$ is a sufficiently general section of a basepoint-free linear system, then  we show the quasi-isomorphism above simplifies to
        $$\pi_{*}\underline{\Omega}^{p}_{Y} \simeq \Omp \oplus \bigoplus_{i=1}^{N-1}\underline{\Omega}^{p}_{X}(\log D)\otimes \cL^{-i},$$
        
        where $\underline{\Omega}^{p}_{X}(\log D)$ is the usual $p^{th}$-graded piece of the Du Bois complex of the pair $(X, D)$ \cite[\S 6]{dubois}. This quasi-isomorphism was first shown by Kov\'acs  \cite[Thm. 6.2.(iv)]{kovacsINJ}.

        The techniques for cyclic coverings have been well known to be applied to prove vanishing statements. As an application of approaching cyclic coverings using mixed Hodge modules, in Section \ref{S5}, we prove the following Kawamata-Viehweg type vanishing theorem for projective, not necessarily smooth, complex varieties. 
        \begin{thm}\label{thm5}
    Let $X$ be an irreducible projective variety of dimension $n$, and $\cL$ a big and nef line bundle on $X$.  If $\lcdef(X)=  \max\{\ell \in \N_{0}| \hspace{.01in}^{p} \cH^{-\ell}(\Q_{X}[n])\neq 0\}$ is the local cohomolgical defect of $X$, then
    $$ H^{i}(X,\underline{\Omega}^{0}_{X} \otimes \cL^{-1}) = 0 \quad \text{for $ i< n - \lcdef(X)$.}$$
\end{thm}
If $X$ is a smooth, irreducible projective variety, then $\Om \simeq \cO_{X}$ and $\lcdef(X) = 0$. Thus, the theorem above extends the well-known vanishing theorem for big and nef line bundles for smooth projective varieties \cite{kawamata} \cite{viehweg}. Note that if $\cL$ were replaced with an ample line bundle in Theorem \ref{thm5}, then one may apply the Kodaira-Saito vanishing theorem \cite{saito2} inductively to prove
    $$ H^{i}(X,\underline{\Omega}^{0}_{X} \otimes \cL^{-1}) = 0 \quad \text{for $ i< n - \lcdef(X)$.}$$
We would also like to note that Kawamata-Viehweg type vanishing results have also been given by Popa \cite{popa}, Wu \cite{Wu}, and Suh \cite{suh}. In \cite{popa}, Popa gives a vanishing theorem for mixed Hodge modules on smooth projective varieties, with restrictions on the base locus of divisors. While both \cite{Wu} and  \cite{suh} give a  Kawamata-Viehweg type vanishing theorem for pure Hodge modules.

\subsection*{Acknowledgments}
I want to thank Donu Arapura,  Mircea Musta\c{t}\u{a},  Wanchun Shen, Sridhar Venkatesh, and Anh Duc Vo for answering my questions while writing this paper, and a special thanks to Lauren\c{t}iu Maxim for all the valuable discussions.

\section{Preliminaries}\label{S1}

A variety will always mean a reduced separated scheme of finite type over $\C.$ We may also take the induced analytic structure and work in the analytic category. Also, a divisor will always mean a Cartier divisor.

Let $X$ be a smooth irreducible variety of dimension $n$. Roughly speaking, a mixed Hodge module $\cM \in MHM(X)$ on $X$ consists of a right regular holonomic $D_{X}$-module $M$ with a good filtration  $F_{\bullet}M$. The filtration $F_{\bullet}M$ is called the Hodge filtration. If one would prefer to work with left $D_{X}$-modules, then we have
$$M^{left} = M \otimes \omega^{-1}_{X}$$
$$F_{\bullet}M^{left} = F_{\bullet -n}M \otimes\omega^{-1}_{X}$$
where $\omega_{X} = \Omega^{n}_{X}$ is the canonical bundle. There is a weight filtration $W_{\bullet}\cM$, where $Gr^{W}_{i}\cM$ is a pure Hodge module of weight $i$ on $X$. There is also a perverse sheaf $K = rat(\cM)$ with $\Q$-coefficients on $X$  such that 
$$ K \otimes \C \simeq DR(\cM): = [M \otimes \wedge^{n} \Theta_{\cX} \rightarrow M \otimes \wedge^{n-1}\Theta_{X} \rightarrow \cdots \rightarrow M][n],$$ 
where $\Theta_{X}$ is the sheaf of vector fields on $X$. This complex is called the de Rham complex of the $D_{X}$-module $M$, and this complex has a filtration induced from the Hodge filtration on $M$,
$$F_{\bullet}DR(\cM): = [F_{\bullet -n}M \otimes \wedge^{n} \Theta_{X} \rightarrow F_{\bullet -n +1}M \otimes \wedge^{n-1}\Theta_{X} \rightarrow \cdots \rightarrow F_{\bullet}M][n],$$

The first condition that is imposed on the category of Hodge modules of weight $w$ on $X$, which is denoted as $HM(X,w)$, is the condition of strict support. For any $\cM \in HM(X,w),$ there is a decomposition
$$ \cM = \bigoplus_{Z} \cM_{Z} \quad \text{for $\cM_{Z} \in HM_{Z}(X,w)$},$$
where $Z \subseteq X$ is a closed irreducible subariety of $X$ and $\cM_{Z}$ is a Hodge module with strick support on $Z$. The theorem below gives a precise characterization of the subcategory $HM_{Z}(X,w)$.
\begin{thm}\cite[Theorem 1.3]{saito4} For any closed irreducible subvariety $Z \subseteq X$, the restriction to sufficiently small open subvarieties of $Z$ induces an equivalence of categories
$$MH_{Z}(X,w) \simeq VHS_{gen}(Z, w - dim(Z))^{p},$$
the right-hand side is the category of polarizable variations of pure Hodge structure of weight $w - dim(Z)$ defined on a smooth, dense open subvariety $U$ of $Z$. Moreover, this equivalence of categories induces a one-to-one correspondence between polarizations of $\cM \in HM_{Z}(X,w)$ and those of the corresponding generic variation of Hodge structure.
\end{thm}
Now consider a divisor $E \subset X$ such that $E_{red}$ has simple normal crossing support and $\mathfrak{j}: V = X \backslash E \hookrightarrow X$ is the natural map. Any polarized variation of Hodge structure $\cH = ((\cE, F_{\bullet}\cE), L)$ of weight $w-n$ on $V$ may be regarded as a pure Hodge module of weight $w$ on $V$ by the theorem above.  For a variation of Hodge structure $\cH = ((\cE, F_{\bullet}\cE), L)$, $\cE$ is the underlying vector bundle, with increasing filtration $F_{\bullet}\cE$ (keeping with Saito's notation), and $L$ is the local system over $\Q$ such that $\cE = \cO_{V} \otimes_{\Q} L.$

There are three natural extensions of $\cH$ to $X,$
$$\mathfrak{j}_{!}\cH \in MHM(X), \quad \mathfrak{j}_{+}\cH \in MHM(X)$$
$$IC_{X}(\cH) = \im \bigg[\mathfrak{j}_{!}\cH\rightarrow \mathfrak{j}_{+}\cH \bigg] \in HM_{X}(X, w).$$
These mixed Hodge modules are defined so that we have
$$DR(\mathfrak{j}_{+(!)}\cH) \simeq \mathfrak{j}_{*(!)}L \otimes \C$$
$$DR(IC^{H}_{X}(\cH)) \simeq IC_{X}(L)\otimes \C.$$
The Hodge filtrations are given by
$$F_{p}(\mathfrak{j}_{+ (!)}\cH) = \displaystyle \sum (\omega_{X} \otimes F_{i}\tilde{\cE}^{\geq -1(>-1)})F_{p-i}D_{X}$$
$$F_{p}(IC^{H}_{X}(\cH)) = \displaystyle \sum (\omega_{X} \otimes F_{i}\tilde{\cE}^{ > -1})F_{p-i}D_{X} \quad \text{in $IC^{H}_{X}(\cH),$}$$
where $\tilde{\cE}^{\geq 0}$ is the lattice of Deligne's regular singular meromorphic extension whose eigenvalues of $res \nabla$ along the irreducible components of $E$ are contained in $[0,1)$. Similarly, $\tilde{\cE}^{> 0}$ is the lattice of Deligne's regular singular meromorphic extension whose eigenvalues of $res \nabla$ along the irreducible components of $E$ are contained in $(0,1]$. These lattices have a natural filtration given by $$F_{\bullet}\tilde{\cE}^{\geq 0 (>1)} = \tilde{\cE}^{\geq 0 (>1)} \cap \mathfrak{j}_{*}F_{\bullet}\cE.$$
An important property of the extension $\mathfrak{j}_{+(!)}\cH$ is given in the proposition below, which was shown by Saito.
\begin{prop}\cite[Prop. 3.11]{saito2}
    There is a quasi-isomorphism of filtered differential complexes
    $$DR(\mathfrak{j}_{+}\cH) \simeq \bigg( (\Omega^{\bullet}_{X}(\log E), F_{\bullet}) \otimes (\tilde{\cE}^{\geq 0}, F_{\bullet}) \bigg)[n]$$
    $$DR(\mathfrak{j}_{!}\cH) \simeq \bigg( (\Omega^{\bullet}_{X}(\log E), F_{\bullet}) \otimes (\tilde{\cE}^{> 0}, F_{\bullet}) \bigg)[n],$$
    where the filtration on $\Omega^{\bullet}_{X}(\log E)$ is given by the ``stupid" filtration. 
\end{prop}

\begin{cor}\label{LogCor}
If the Hodge filtration for $\cH$ is trivial, 
$$F_{i}\cE =\begin{cases}
 \cE & \text{for $i\geq 0$}\\ \\
0 & \text{otherwise}
\end{cases}$$
 then we have the quasi-isomorphisms
$$Gr^{F}_{-p}DR(\mathfrak{j}_{+}\cH) \simeq (\Omega^{p}_{X}(\log E) \otimes \tilde{\cE}^{\geq 0})[n-p]$$
$$Gr^{F}_{-p}DR(\mathfrak{j}_{!}\cH) \simeq (\Omega^{p}_{X}(\log E) \otimes \tilde{\cE}^{> 0})[n-p].$$
\end{cor}

Suppose $Z$ is a singular variety and $Z$ embeds into the smooth variety $X$. If $$i:Z \hookrightarrow X$$
 is the natural inclusion map, then we have an equivalence of categories:
$$ \bT i_{+}: D^{b}MHM(Z) \arrow[r] & D^{b}MHM_{Z}(X), 
\eT$$
where $D^{b}MHM_{Z}(X)$ is the full subcategory of $D^{b}MHM(X)$ whose objects have cohomological supports in $Z$ \cite{saito2}. For any $\cN \in D^{b}MHM(Z)$, the complex $Gr^{F}_{p}DR(i_{+}\cN) \in D^{b}_{coh}(\cO_{X})$ is a well defined complex in $D^{b}_{coh}(\cO_Z),$ and is independent of the embedding $i:Z \hookrightarrow X$. In general, we can always take an affine covering $\{Z_{\alpha}\}_{\alpha \in I}$ of $Z$ and take a closed embeddig $i_{\alpha +}:Z_{\alpha} \hookrightarrow X_{\alpha}$ and study mixed Hodge modules locally on $Z$. 
\begin{thm}\label{Saito'sMainThm}\cite{saito2}
    We have natural functors $f_{+}, f_{!}, f^{*}, f^{!}, \psi_{g}, \phi_{g,1}, \D, \boxtimes, \otimes$ and $\cH om$ between $D^{b}MHM(Z)$ such that these functors are compatible with the corresponding functors on the underlying $\Q$-complexes via:
    $$rat: D^{b}MHM(Z) \rightarrow D^{b}Per(\Q_{Z}) \rightarrow D^{b}_{c}(Z)$$
    where $f$ is a morphism of algebraic varieties, $g \in \Gamma(X, \cO_{Z})$ and $\phi_{g,1} = Ker(T_{s}-1)$ with $T_{s}$ the semi-simple part of the monodromy $T$ of $\phi_{g}.$
\end{thm}

\begin{thm}\label{direct}\cite[Thm. 0.1]{saito5}
Let $f: Z \rightarrow Z'$ be a proper map between algebraic varieties. For every $p \in \Z$, one has a natural isomorphism of functors
$$\bR f_{\ast} \circ Gr^{F}_{p}DR(\bullet) \simeq Gr^{F}_{p}DR \circ f_{+}(\bullet),$$
as functors from $D^{b}(MHM(Z))$ to  $D^{b}_{coh}(\cO_{Z'})$. 
\end{thm}

\begin{prop}[Proper Base Change]\cite[(4.4.3)]{saito2}
Let
$$\bT Z' \arrow[d,"f'"] \arrow[r, "g"] & Z \arrow[d, "f"]\\
Y' \arrow[r, "g'"] & Y \eT$$
be a cartesian diagram of algebraic varieties. Then 
$$(g')^{*}f_{!} \simeq f^{'}_{!}g^{*} \quad \text{in $D^{b}MHM(Y')$}$$
and is compatible with proper base change in $D^{b}_{c}(Y')$
\end{prop}

\section{Proof Of The Decomposition Theorem}\label{S2}
  Before proving the theorem in the general case, let's begin with a smooth quasi-projective variety $X$ of dimension $n$ and $\cA$ an ample line bundle on $X$. We will follow the construction of cyclic coverings given by \cite{EsnaultViehweg}. For sufficiently large $N$, we have a smooth divisor $H$ such that $\cO_{\cX}(H) = \cA^{N}$ and there exists a section $s \in H^{0}(X, \cA^{N})$ such that
$$(s) + N \cdot A = H$$ 
where $A$ is a divisor such that $\cA = \cO_{X}(A)$. The normalization $\cY$ of $X$ in the field extension $\C(X)(\sqrt[N]{s})$ is the $N$-fold cyclic covering of $X$. Let $\Pi: \cY \rightarrow X$ denote the surjective map of $\cY$ onto $X$. Since $H$ is smooth, $\cY$ is also smooth. If $H'= \Pi^{-1}(H)_{red}$, then the induced map $\Pi: H' \rightarrow H$ is an isomorphism.  There is also a natural decomposition of the structure sheaf,
$$\Pi_{*}\cO_{\cY} \cong \cO_{X} \oplus \bigoplus_{i=1}^{N-1}\cA^{-i}.$$ 
Choose local coordinates $x_{1}, x_{2}, \cdots, x_{m}$ for $X$ such that $H$ is defined by $x_{1} = 0$.  Then
$$y_{1} = \sqrt[N]{x_{1}} \quad \text{and} \quad  \text{$y_{i}= x_{i}$ for $i \geq 2$}$$
are local coordinates for $\cY.$ 
In these local coordinates, 
$$\Pi_{*}\cO_{\cY} \cong \cO_{X} \oplus \bigoplus_{i=1}^{N-1}(\cO_{X} \cdot y_{1}^{i}).$$
The differential map $d: \cO_{\cY} \rightarrow \Omega^{1}_{\cY}$ induces a logarithmic connection
$$\nabla_{i}: \cA^{-i} \rightarrow \Omega^{1}_{X}(\log H) \otimes \cA^{-i}$$
$$\nabla_{i}(gy^{i}) = (dg + \frac{i}{N}g \frac{dx_{1}}{x_{1}})y^{i} \quad \text{for $g \in \cO_{X}.$}$$
 If $V = X \backslash H$, then restricting $\nabla_{i}$ to $V$ is a connection on $V$. In local coordinates, the underlying locally constant sheaf of the connection, in $\C$ coefficients, is given by
$$\bigg \{ g \in \cO^{an}_{V}| \hspace{.05in}\frac{\partial g}{\partial x_{1}} = -\frac{i}{N} \frac{g}{x_{1}}, \quad \frac{\partial g}{ \partial x_{i}}=0 \quad \text{for $i>1$} \bigg \}= \C \cdot \frac{1}{\sqrt[N]{x^{i}}}.$$
We can also describe the cyclic covering in terms of mixed Hodge modules. If $V' = \cY \backslash H'$, then the induced map $\Pi: V' \rightarrow V$ is finite \'etale. If we consider the constant variation of Hodge structure $\Q^{H}_{V'}[n] \in HM_{V'}(V',n)$, then 
$$\Pi_{+}\Q^{H}_{V'}[n] \cong \Q^{H}_{V}[n] \oplus \tilde{\cV}$$
where $\tilde{\cV} \in HM_{V}(V,n)$ is a variation of Hodge structure. The underlying locally free sheaf of $\tilde{\cV}$ is given by 
$$\mathcal{E} = \displaystyle \bigoplus_{i=1}^{N-1} \cA^{-i}\vert_{V} \cong  \bigoplus_{i=1}^{N-1}\cO_{V},$$
and the connection 
$$\nabla: \mathcal{E} \rightarrow \Omega^{1}_{V} \otimes \mathcal{E},$$ 
is determined by the connections $\nabla_{i}$ described above. Note that for $\Pi_{+}\Q^{H}_{V'}[n]$ to have a rational structure, we generally cannot consider the connections $\nabla_{i}$ individually. Also, since the Hodge filtration for $\Q^{H}_{V'}[m]$ is trivial, the Hodge filtration for $\tilde{\cV}$ is also trivial. 

If we consider the $N$-fold cyclic covering $\Pi: \cY \rightarrow X$, then we have the following decomposition,
$$\Pi_{+}\Q^{H}_{\cY}[n] \cong \Q^{H}_{X}[n] \oplus IC_{X}(\tilde{\cV}),$$
where $IC_{X}(\tilde{\cV}) \in HM_{X}(X,n)$ is the unique extension of the variation of Hodge structure $\tilde{\cV}.$ Let $\mathfrak{j}: V \hookrightarrow X$, $\rho: V' \hookrightarrow \cY$, $\mathfrak{i}: H \hookrightarrow X$ and $\iota: H' \hookrightarrow \cY$ denote the inclusion maps. Then we have the morphism of triangles in $MHM(X)$,
$$\bT \Pi_{+}\rho_{!}\Q^{H}_{V'}[n] \ar[r] & \Pi_{+}\Q^{H}_{\cY}[n] \ar[r] & \Pi_{+}(\iota_{+}\Q^{H}_{H'}[n-1])[1] \arrow[r, "+1"]  & \hfill\\
\mathfrak{j}_{!}\Q^{H}_{V}[n] \ar[r] \ar[u] & \Q^{H}_{X}[n] \ar[r] \ar[u] & (\mathfrak{i}_{+}\Q^{H}_{H}[n-1])[1] \arrow[r, "+1"] \ar[u, "\rotatebox{90}{$\sim$}"] & \hfill \eT$$
From our discussion, the previous diagram can be written as
$$\bT \mathfrak{j}_{!}\Q^{H}_{V}[n] \oplus \displaystyle \mathfrak{j}_{!}\tilde{\cV} \ar[r] & \Q^{H}_{X}[n] \oplus \displaystyle IC_{X}(\cV) \ar[r] & \Pi_{+}(\iota_{+}\Q^{H}_{H'}[n-1])[1] \arrow[r, "+1"]  & \hfill\\
\mathfrak{j}_{!}\Q^{H}_{V}[n] \ar[r] \ar[u] & \Q^{H}_{X}[n] \ar[r] \ar[u] & (\mathfrak{i}_{+}\Q^{H}_{H}[n-1])[1] \arrow[r, "+1"] \ar[u, "\rotatebox{90}{$\sim$}"] & \hfill \eT$$
We see that there is an identification $\mathfrak{j}_{!}\tilde{\cV} \cong IC_{\cX}(\tilde{\cV}).$
Thus,
 $$\Pi_{+}\Q^{H}_{\cY}[n] \cong \Q^{H}_{X}[n] \oplus \mathfrak{j}_{!}\tilde{\cV}.$$
 Another way to see that there is an identification  $\mathfrak{j}_{!}\tilde{\cV}\cong IC_{\cX}(\tilde{\cV}),$ the locally free sheaf $\cA^{-i}$ is a lattice of $\mathfrak{j}_{*}(\cA^{-i}|_{V}).$ The line bundle $\cA^{-i}$ is precisely the lattice for $\mathfrak{j}_{*}(\cA^{-i}|_{V})$  whose eigenvalues of the residue $res\nabla$ are between $(0,1]$. Note, the line bundle $\cA^{-i}$ is also the lattice for $\mathfrak{j}_{*}(\cA^{-i}|_{V})$  whose eigenvalues of the residue $res \nabla$ are between $[0,1)$. 
 
 By Corollary \ref{LogCor},  we have quasi-isomorphisms
$$Gr^{F}_{-p}DR(j_{!}\tilde{\cV})  \simeq \displaystyle \bigoplus_{i=1}^{N-1}(\Omega^{p}_{X}(\log H) \otimes \cA^{-i})[n-p].$$
Thus, we obtain the well-known isomorphism
$$ \Pi_{*}\Omega^{p}_{\cY} \cong \Omega^{p}_{X} \oplus \bigoplus_{i=1}^{N-1} (\Omega^{p}_{X}(\log H) \otimes \cA^{-i}) .$$

  Now let $\cL$ be an arbitrary line bundle on a variety $X$, possibly singular, such that for a positive integer $N \geq 0$, there exists a global section
    $$s \in \Gamma(X, \cL^{N})$$
    that defines an effective divisor $D.$ Let $\pi:Y \rightarrow X$ denote the $N$-fold cyclic covering of $D.$ Since $X$ and $D$ can be arbitrary, $Y$ may be singular. 
    \begin{rmk}
        We are not taking the normalization of $Y,$ which differs from \cite{EsnaultViehweg}. Our definition of the $N$-fold cyclic covering of $D$ aligns with the construction given in \cite[\S 4.1.B]{LararsfeldI}.
    \end{rmk} 
    If we set $D':= \pi^{-1}(D) \subset Y$, then the induced map $\pi:D'_{red} \rightarrow D_{red}$ is an isomorphism. Also, the induced map $\pi: Y \backslash D' \rightarrow X \backslash D$ is a finite \'etale map.  For finite \'etale maps, we need the following lemma.
    \begin{lemma}
        If $f: Z' \rightarrow Z$ is a finite \'etale map, then there exists $\cM \in D^{b}MHM(Z)$ such that 
        $$f_{+}\Q^{H}_{Z'}[d] \simeq \Q^{H}_{Z}[d] \oplus \cM.$$
    \end{lemma}

    \begin{proof}
        Let $k$ denote the degree of the map $f$. Since $f$ is finite, $f_{+} = f_{!}.$ Also, since $f$ is \'etale, $f^{!} = f^{*}.$ From this observation, we have the adjunction maps
        $$\Q^{H}_{Z}[d] \rightarrow f_{+}f^{*}\Q^{H}_{Z}[d] = f_{+}\Q^{H}_{Z'}[d]$$
        $$f_{+}\Q^{H}_{Z'}[d] = f_{!}f^{*}\Q^{H}_{Z}[d] = f_{!}f^{!}\Q^{H}_{Z}[d] \rightarrow \Q^{H}_{Z}[d].$$
        The composition of the preceding maps is identified with the multiplication by the degree of the map $f,$
        $$\bT \Q^{H}_{Z}[d] \arrow[rr, bend left, "k"] \arrow[r] & f_{+}\Q^{H}_{Z'}[d] \arrow[r] & \Q^{H}_{Z}[d] \eT$$
        Indeed, if $Z$ is smooth, then this is well known. Let $Z_{reg}$ denote the smooth set of $Z$. Then there is an inclusion 
        $$Hom(\Q^{H}_{Z}[d],\Q^{H}_{Z}[d]) \subseteq Hom(\Q^{H}_{Z_{reg}}[d], \Q^{H}_{Z_{reg}}[d]).$$
        Because the map $\Q^{H}_{Z}[d] \rightarrow  f_{+}\Q^{H}_{Z'}[d] \rightarrow \Q^{H}_{Z}[d]$ agrees with the multiplication by $k$ on the smooth locus, the inclusion above implies it must be the multiplication map. Therefore, the natural map $\Q^{H}_{Z'}[d] \rightarrow f_{+}\Q^{H}_{Z}[d]$ splits. So, there exists $\cM \in D^{b}MHM(Z)$ such that 
        $$f_{+}\Q^{H}_{Z'}[d] \simeq \Q^{H}_{Z}[d] \oplus \cM.$$
    \end{proof}
    \begin{rmk}
        A more general discussion of the trace morphism in $D^{b}MHM(X)$ is given by \cite[Lemma 5.3]{CMSS}, and recently in \cite{DORdescent}.
    \end{rmk}
        If we set $U' = Y \backslash D'$ and $U = X \backslash D$, then the finite \'etale map $\pi: U' \rightarrow 
      U$ induces an isomorphism 
      $$\pi_{+}\Q^{H}_{Y}[n]\vert_{U} \simeq \Q^{H}_{U}[n] \oplus \cV \quad \text{for some $\cV \in D^{b}MHM(U).$}$$
      If $j:U \hookrightarrow X$ is the inclusion map, then the isomorphism above gives a natural map
    $$\Q^{H}_{X}[n] \oplus j_{!}\cV \rightarrow \pi_{+}\Q^{H}_{Y}[n].$$
    This map is compatible with proper base change, and thus we have the morphism of triangles,
    $$\bT j_{!}j^{*}(\pi_{+}\Q^{H}_{Y}[n]) \ar[r] & \pi_{+}\Q^{H}_{Y}[n] \ar[r] & i_{+}i^{*}(\pi_{+}\Q^{H}_{Y}[n]) \arrow[r, "+1"]  & \hfill\\
j_{!}\Q^{H}_{U}[n] \oplus j_{!}\cV \ar[r] \ar[u, "\rotatebox{90}{$\sim$}"] & \Q^{H}_{X}[n] \oplus j_{!}\cV \ar[r] \ar[u] & (\Q^{H}_{D_{red}}[n-1]) [1] \arrow[r, "+1"] \ar[u, "\rotatebox{90}{$\sim$}"] & \hfill \eT$$
    The outer morphisms are isomorphisms. Therefore, the map $\Q^{H}_{X}[n] \oplus j_{!}\cV \rightarrow \pi_{+}\Q^{H}_{Y}[n]$ is also an isomorphism.  Applying the functor $Gr^{F}_{-p}DR(\bullet)$, and using that the map $\pi:Y \rightarrow X$ is affine and proper \ref{direct}, for $0 \leq p \leq n$, we have quasi-isomorphisms,
    $$\pi_{*}\OmpY \simeq Gr^{F}_{-p}DR(\pi_{+}\Q^{H}_{Y}[n])[p-n]$$
    $$\simeq Gr^{F}_{-p}DR(\Q^{H}_{X}[n])[p-n] \oplus Gr^{F}_{-p}DR(j_{!}\cV)[p-n] \simeq \underline{\Omega}^{p}_{X} \oplus Gr^{F}_{-p}DR(j_{!}\cV)[p-n].$$
The quasi-isomorphism $Gr^{F}_{-p}DR(\Q^{H}_{Y}[n])[p-n] \simeq \Omp$ is due to Saito \cite{saito5}, where $\Omp$ is the $p^{th}$ graded piece of the Du-Bois complex (up to a shift) \cite{dubois}.
      \begin{thm}\label{MainThm1}
     Let $\cL$ be an arbitrary line bundle on an algebraic variety $X$ of dimension $n$. Suppose that for a positive integer $N$ there exists a global section $s \in \Gamma(X, \cL^{N})$
    that defines an effective Cartier divisor $D.$ Let $\pi: Y \rightarrow X$ denote the $N$-fold cyclic covering of $D$.  If $j: U= X \backslash D \hookrightarrow X$ is the natural map, then there exists $\cV \in D^{b}MHM(U)$  such that we have the following isomorphisms,
    $$\pi_{+}\Q^{H}_{Y}[n] \simeq \Q^{H}_{X}[n] \oplus j_{!}\cV.$$
    For $0 \leq p \leq n=\dim(X)$, if we apply the functor $Gr^{F}_{-p}DR(\bullet):D^{b}MHM(X) \rightarrow D^{b}_{coh}(\cO_{X})$ to the isomorphism above, we obtain the quasi-isomorphism
    $$\pi_{*}\OmpY  \simeq \underline{\Omega}^{p}_{X} \oplus Gr^{F}_{-p}DR( j_{!}\cV)[p-n].$$
    \end{thm}

    \begin{set}\label{setting}
        For the rest of this paper, we fix the notation of $\cL$ being an arbitrary line bundle on a variety $X$ of dimension $n$ such that for a positive integer $N $, there exists a global section
    $$s \in \Gamma(X, \cL^{N})$$
    that defines an effective Cartier divisor $D.$ We set $U = X \backslash D$ with natural map $j:U \hookrightarrow X.$
    \end{set}

    \section{Cyclic Coverings For Smooth Varieties}\label{S3}

     For this section,  $X$ will be a smooth irreducible variety of dimension $n$. With $X$ smooth and irreducible, $\Q^{H}_{X}[n] \in HM_{X}(X,n)$ is a pure Hodge module. We are in the setting of \ref{setting}.   Thus, suppose that for a positive integer $N$ there exists a global section $s \in \Gamma(X, \cL^{N})$
    that defines an effective Cartier divisor $D.$ Let $\pi: Y \rightarrow X$ denote the $N$-fold cyclic covering of $D$.
    
    From our main theorem, we have the decomposition
     $$\pi_{+}\Q^{H}_{Y}[n] \cong \Q^{H}_{X}[n] \oplus j_{!}\cV \quad \text{for $\cV \in D^{b}MHM(U)$.}$$
     It is important to note that since $X$ is smooth, $\cV \in HM_{U}(U,n)$ is a polarized variation of Hodge structure on the smooth variety $U = X \backslash D.$

In the derived category of mixed Hodge modules $D^{b}MHM(X)$, we have the duality functor 
$$\D: D^{b}MHM(X) \rightarrow D^{b}MHM(X)^{op}.$$
Since the map $\pi:Y \rightarrow X$ is projective, the theorem above implies the following isomorphisms
$$\pi_{+}\D(\Q^{H}_{Y}[n]) \simeq \D(\pi_{+}\Q^{H}_{Y}[n]) \simeq \D\bigg( \Q^{H}_{X} \oplus j_{!}\cV \bigg) \simeq \D(\Q^{H}_{X}[n]) \oplus \D(j_{!}\cV) \simeq \D(\Q^{H}_{X}[n]) \oplus j_{+}(\D(\cV)).$$
 With $X$ be being smooth, we have isomorphisms $\D(\Q^{H}_{X}[n]) \simeq \Q^{H}_{X}(n)$ and $\D(\cV) \cong \cV(n)$. Thus, there is an isomorphism
     $$\pi_{+}\D(\Q^{H}_{Y}[n])(-n) \simeq \Q^{H}_{X}[n] \oplus j_{+}\cV.$$
With the duality functor $\D$, we have the duality functor on $D^{b}_{coh}(\cO_{X}),$
     $$\cD: D^{b}_{coh}(\cO_{X}) \rightarrow  D^{b}_{coh}(\cO_{X})^{op}$$
     $$\cD(\cF) := \bR \cH om_{\cO_{X}}(\cF, \omega_{X}) \quad \text{for $\cF \in D^{b}_{coh}(\cO_{X}).$}$$
     Since $X$ is smooth, we set $\omega_{X} = \Omega^{n}_{X}$. For any $\cM \in D^{b}MHM(X)$, by \cite[\S 2.4.4]{saito},
     $$\cD(Gr^{F}_{-p}DR(\cM))[n] \simeq Gr^{F}_{p}DR(\D(\cM)).$$
     Thus,
     $$\pi_{*}\cD(\underline{\Omega}^{n-p}_{Y}) \cong \pi_{*}\cD(Gr^{F}_{p-n}DR(\Q^{H}_{Y}))[p] \cong Gr^{F}_{-p}DR \bigg(\pi_{+}\D(\Q^{H}_{Y}[n])(-n) \bigg)[p-n]$$
     $$\simeq Gr^{F}_{-p}DR\bigg(\Q^{H}_{X}[n] \oplus j_{+}\cV \bigg)[p-n] \simeq \Omega^{p}_{X}  \oplus Gr^{F}_{-p}DR(j_{+}\cV)[p-n].$$
     Furthermore, there is a natural identification
      $$\pi_{+}IC^{H}_{Y} \cong  \Q^{H}_{X}[n] \oplus IC_{X}(\cV)  \cong  \Q^{H}_{X}[n] \oplus  \im [j_{!}\cV \rightarrow j_{+}\cV]. $$
      In the case when $\Q^{H}_{Y}[n] \cong IC^{H}_{Y}$, the variety $Y$ is said to be a \emph{rational homology manifold}.

\subsection{Simple Normal Crossing Support}
First, consider the case where $D$  has simple normal crossing support. That is, $E = D_{red}$ is a simple normal crossing divisor. Assume $D = a_{1}E_{1}+a_{2}E_{2}+ \cdots +a_{\ell}E_{\ell}$, where $E_{i}$ are the irreducible components. Let $\pi: Y \rightarrow X$ be the $N$-fold cyclic covering of  $D$.  From our main theorem, we have the decomposition
     $$\pi_{+}\Q^{H}_{Y}[n] \cong \Q^{H}_{X}[n] \oplus j_{!}\cV \quad \text{for $\cV \in HM_{U}(U,n)$.}$$ If $U = X \backslash D$ and $U' = Y \backslash \pi^{*}D$, then the map $\pi:U' \rightarrow U$ is finite \'etale, and there is an isomorphism
     $$\pi_{*}\cO_{U'} \cong \cO_{U} \oplus \bigg( \bigoplus_{i=1}^{N-1} \cL^{-i}|_{U} \bigg).$$
     The locally free sheaf  $\bigg( \displaystyle \bigoplus_{i=1}^{N-1} \cL^{-i}|_{U} \bigg) $ is the left $D_{U}$-module of the pure Hodge module $\cV$, whose Hodge filtration is trivial. The differential map $\cO_{U'} \rightarrow \Omega^{1}_{U'}$ induces the connection
      $$\bigg( \displaystyle \bigoplus_{i=1}^{N-1} \cL^{-i}|_{U} \bigg) \rightarrow \Omega^{1}_{U}\otimes \bigg( \displaystyle \bigoplus_{i=1}^{N-1} \cL^{-i}|_{U} \bigg).$$
      The connection splits into individual connections
      $ \cL^{-i}|_{U} \rightarrow \Omega^{1}_{U}\otimes  \cL^{-i}|_{U}  \quad \text{for $1 \leq i \leq N-1$}.$
      If $L_{i}$ denotes the $\C$-local system defined by the connection $ \cL^{-i}|_{U} \rightarrow \Omega^{1}_{U}\otimes  \cL^{-i}|_{U},$ then $\cH_{i} = ((\cL^{-i}, F_{\bullet}\cL^{-i}), L_{i})$ is a filtered regular holonomic $D_{U}$-module, but does not in general have a $\Q$-structure.     This connection can be extended to a logarithmic connection 
       $$ \cL^{-i} \rightarrow \Omega^{1}_{X}(\log E)\otimes  \cL^{-i}  \quad \text{for $1 \leq i \leq N-1$.}$$
      Locally, the divisor $D= \{ux_{1}^{a_{1}}x_{2}^{a_{2}}\cdots x_{\ell}^{a_{\ell}}= 0\}$, where $u$ is a unit. The logarithmic connection of $\cL^{-i} = \cO_{X}t^{i}$ is given by
      $$\nabla(gt^{i}) = \bigg(dg + \dfrac{i}{N}\dfrac{du}{u}g+\sum_{j=1}^{\ell}\frac{ia_{j}}{N}\frac{dx_{j}}{x_{j}}g\bigg)t^{i} \quad \text{ for $g \in \cO_{X}$.}$$
       To ensure the eigenvalues of $res \nabla$ along the irreducible components of $E$ lie in $(0,1]$, we need to twist the connection appropriately. Let 
      \begin{equation}\label{m_{ij}}
      m_{ij} = \begin{cases}
    
      \lfloor (ia_{j})/N\rfloor & \text{if $(ia_{j})/N \notin \Z$} \\ \\
       (ia_{j})/N-1 & \text{if $(ia_{j})/N \in \Z .$}
      \end{cases}
      \end{equation}
      Then, for $1 \leq i \leq N-1$, the logarithmic connection 
      $$\bigg(\bigotimes_{j=1}^{\ell}\cO_{X}(m_{ij}E_{j})\bigg) \otimes \cL^{-i} \rightarrow \bigg(\Omega^{1}_{X}(\log E) \otimes \bigotimes_{j=1}^{\ell}\cO_{X}(m_{ij}E_{j})\bigg)\otimes \cL^{-i} $$
      has eigenvalues between $(0,1].$ In order  to have the eigenvalues of $res \nabla$ along the irreducible components of $E$ to be contained in $[0,1),$ let 
      \begin{equation}\label{m'_{ij}}
      m'_{ij} = 
      \lfloor (ia_{j})/N\rfloor 
      \end{equation}
      Then, for $1 \leq i \leq N-1$, the logarithmic connection 
      $$\bigg(\bigotimes_{j=1}^{\ell}\cO_{X}(m'_{ij}E_{j})\bigg) \otimes \cL^{-i} \rightarrow \bigg(\Omega^{1}_{X}(\log E) \otimes \bigotimes_{j=1}^{\ell}\cO_{X}(m'_{ij}E_{j})\bigg)\otimes \cL^{-i}$$
      has eigenvalues between $[0,1).$ As a consequence of Corollary \ref{LogCor}, we have quasi-isomorphisms
      $$Gr^{F}_{-p}DR(j_{!}\cH_{i})[p-n] \simeq \bigg(\Omega^{p}_{X}\log(E)\otimes \bigotimes_{j=1}^{\ell}\cO_{X}(m_{ij}E_{j})\bigg)\otimes \cL^{-i}$$
      $$Gr^{F}_{-p}DR(j_{+}\cH_{i})[p-n] \simeq  \bigg(\Omega^{n-p}_{X}\log(E)\otimes \bigotimes_{j=1}^{\ell}\cO_{X}(m'_{ij}E_{j})\bigg)\otimes\cL^{-i}.$$

      $$Gr^{F}_{-p}DR(\pi_{+}\Q^{H}_{Y}[n]) \simeq Gr^{F}_{-p}DR(\Q^{H}_{X}[n]) \oplus Gr^{F}_{-p}DR(j_!\cV) \simeq Gr^{F}_{-p}DR(\Q^{H}_{X}[n]) \oplus \bigg( \bigoplus_{i=1}^{N-1}Gr^{F}_{-p}DR(j_{!}\cH_{i})\bigg)$$      $$\pi_{*}\OmpY \simeq \Omega^{p}_{X} \oplus \bigoplus^{N-1}_{i=1} \bigg(\Omega^{p}_{X}\log(E)\otimes \bigotimes_{j=1}^{\ell}\cO_{X}(m_{ij}E_{j})\bigg)\otimes \cL^{-i}.$$
      Similarly, we also have the quasi-isomorphism
      $$\pi_{*}\cD(\OmpY) \simeq \Omega^{n-p}_{X} \oplus \bigoplus^{N-1}_{i=1} \bigg(\Omega^{n-p}_{X}\log(E)\otimes \bigotimes_{j=1}^{\ell}\cO_{X}(m'_{ij}E_{j})\bigg)\otimes\cL^{-i}.$$
 Notice that these quasi-isomorphisms give us the following proposition.
      \begin{prop}
    If $\pi:Y \rightarrow X$ is the $N$-fold cyclic covering $D$, where $D$ has simple normal crossing support, then 
    $$\cH^{i}(\cD(\underline{\Omega}^{n-p}_Y))=0 \quad \text{for $p \in \Z$ and $i>0$.}$$
    $$\cH^{i}(\underline{\Omega}^{p}_Y)=0 \quad \text{for $p \in \Z$ and $i>0$.}$$
\end{prop}

   \begin{rmk}
       If $n:Y'\rightarrow Y$ is the normalization of $Y$, then 
       $$n_{*}\cO_{Y'} \cong \cH^{0}\cD(\underline{\Omega}^{n}_{Y}).$$
       Therefore,
       $$\pi_{*}n_{*}\cO_{Y'} \cong \cO_{X} \oplus \bigoplus^{N-1}_{i=1} \bigg(\bigotimes_{j=1}^{\ell}\cO_{X}(m'_{ij}E_{j})\bigg)\otimes\cL^{-i},$$
which matches precisely with \cite[Cor. 3.11]{EsnaultViehweg} in the case when $D$ has simple normal crossing support. 
   \end{rmk}

  \begin{ex}
      Consider the Whitney umbrella,
      $$\pi: Y = \Spec  \bigg( \C[x,y,t]/\langle t^2  -x^2y \rangle\bigg) \rightarrow \Spec \bigg(\C[x,y] \bigg) = 
      X.$$
      We  have a decomposition
       $$\pi_{*}\cO_{Y} = \cO_{X} \oplus \cL^{-1},$$
      where $\cL^{-1} = \cO_{X}t$. There is a logarithmic connection,
      $$\nabla(gt) = \bigg(dg + g\bigg(\frac{dx}{x} + \frac{1}{2} \frac{dy}{y}\bigg)\bigg)t \quad \text{for $g \in \cO_{X}.$}$$
      Notice that the eigenvalues of $res \nabla$ along the irreducible components are between $(0,1].$ Let $H$ denote the divisor $x = 0$ for $X$. If we twist the connection by $H$, then the eigenvalues of $res \nabla$ along the irreducible components  are between $[0,1),$
      $$\nabla\bigg(g\frac{t}{x}\bigg) = \bigg(dg + \frac{1}{2}g \frac{dy}{y}\bigg)\frac{t}{x}\quad \text{for $g \in \cO_{X}.$}$$
       If we set $E$ to be the divisor $xy=0$, then 
      $$\pi_{*}\underline{\Omega}^0_{Y} \cong \cO_{X} \oplus \cL^{-1} \cong \pi_{*}\cO_{Y}, \quad  \pi_{*}\underline{\Omega}^1_{Y} \cong \Omega^1_{X} \oplus \Omega^{1}_{X}(\log E) \otimes \cL^{-1}, \quad \pi_{*}\underline{\Omega}^2_{Y} \cong \Omega^2_{X} \oplus \Omega^{2}_{X}(E) \otimes \cL^{-1} $$
      $$\pi_{*}\cD(\underline{\Omega}^2_{Y}) \cong \cO_{X} \oplus (\cO_{X}(H) \otimes \cL^{-1}) \quad \pi_{*}\cD(\underline{\Omega}^1_{Y}) \cong \Omega^1_{X} \oplus (\Omega^{1}_{X}(\log E)\otimes\cO_{X}(H) \otimes \cL^{-1}),$$ 
      $$\pi_{*}\cD(\underline{\Omega}^0_{Y}) \cong \Omega^2_{X} \oplus (\Omega^{2}_{X}(E)\otimes\cO_X(H) \otimes \cL^{-1}) $$
  \end{ex}

  \begin{prop}\label{sncThm}
      Assume $D = a_{1}E_{1}+a_{2}E_{2}+ \cdots +a_{\ell}E_{\ell}$ is an effective divisor with simple normal crossing support, where $E_{i}$ are the irreducible components. If $\pi:Y \rightarrow X$ is the $N$-fold cyclic covering of  $D$, then  $\Q_{Y}[n] \cong IC_{Y}^{H}$ if and only if $gcd(N,a_{j})=1$ for $1 \leq j \leq \ell.$
  \end{prop}

  \begin{proof}
      Recall that we have the following decompositions,
      $$\pi_{+}\Q^{H}_{Y}[n] \simeq \Q^{H}_{X}[n] \oplus j_{!}\cV$$
        $$\pi_{+}\D(\Q^{H}_{Y}[n])(-n) \simeq \Q^{H}_{X}[n] \oplus j_{+}\cV.$$
      $$\pi_{+}IC^{H}_{Y} \cong  \Q^{H}_{X}[n] \oplus IC_{X}(\cV)  \cong  \Q^{H}_{X}[n] \oplus  \im[j_{!}\cV \rightarrow j_{+}\cV].$$
      We also have $m_{ij}$ and $m'_{ij}$ defined in \ref{m_{ij}} and $\ref{m'_{ij}},$ respectively. If $\Q^{H}_{Y}[n] \cong IC^{H}_{Y}$, then $j_{!}\cV \cong j_{*}\cV$, and it is clear that $m_{ij}= m'_{ij}$. If $m_{ij}= m'_{ij}$, then  we must have
        $$\pi_{*}\OmpY \simeq \pi_{*}\cD(\underline{\Omega}^{n-p}_{Y})\quad \text{for $p \in \Z$,}$$
        which implies $j_{!}\cV = j_{+}\cV$. Thus
        $$\pi_{+}\Q^{H}_{Y}[n] \cong \pi_{+}IC^{H}_{Y} \cong \pi_{+}\D(\Q^{H}_{Y}[n])(-n),$$
        which then implies
        $$\Q^{H}_{Y}[n] \cong IC^{H}_{Y} \cong \D(\Q^{H}_{Y}[n])(-n).$$
        Thus, we have $\Q^{H}_{Y}[n] \cong IC^{H}_{Y}$ if and only if $m_{ij} = m'_{ij}$. This occurs if and only if 
      $$(ia_{j})/N \notin \Z \quad \text{for $1 \leq j \leq \ell$ and $1 \leq i \leq N-1$}$$
      $$ia_{j} \not\equiv 0 (\bmod{N}) \quad \text{for $1 \leq j \leq \ell$ and $1 \leq i \leq N-1$}$$
      $$gcd(N,a_{j})=1 \quad \text{for $1 \leq j \leq \ell$.}$$
      
  \end{proof}

  \subsection{General Case}
   Now we consider any divisor $D$ on $X.$ We can take a strong log resolution $\phi:\tX \rightarrow X$ of the pair $(X, D)$. That is, $\phi:\tX \rightarrow X$ is a projective morphism where $\tX$ is smooth, $E= \phi^{*}D_{red}$  is a simple normal crossing divisor, and the induced map $\phi: \tX \backslash E \rightarrow X \backslash D$ is an isomorphism. We have the commutative diagram
$$ \bT \tY \arrow[d, "\tilde{\pi}"] \arrow[r]& Y  \ar[d, "\pi"] \\
  \tX \arrow[r, "\phi"] & X \eT$$
  where $\tY = (Y \times_{X}\tX)_{red}$. Under this construction, we have $\tY = \Spec \bigg(\displaystyle \bigoplus_{i=0}^{N-1}\phi^{*}\cL^{-i}\bigg)$ to be the $N$-fold cyclic covering of  $\phi^{*}D$. Furthermore, if $\varrho: \tX \backslash E \hookrightarrow \tX$ is the natural inclusion, by proper base change, we have the isomorphism
$$\tilde{\pi}_{+}\Q^{H}_{\tY}[n] \cong \Q^{H}_{\tX}[n] \oplus \varrho_{!}\cV,$$
where we identify $\cV \in HM_{\tX \backslash E}(\tX \backslash E, n)$ via the isomorphism $\phi: \tX \backslash E \rightarrow X \backslash D$.
\begin{thm}\label{SmoothDecomp}
    If $Y \rightarrow X$ is this $N$-fold cyclic covering of $D$ and $\phi:\tX \rightarrow X$ is a log resolution of the pair $(X,D)$ with $E= \phi^{*}D_{red}$, then for $1 \leq i \leq N-1$, there exists effective divisors $D^{> 0}_{i}$ and $D^{ \geq0}_{i}$ on $\tX$ such that 
    $$\pi_{*}\OmpY \simeq \Omega^{p}_{X} \oplus \bigoplus^{N-1}_{i=1} \bR \phi_{*}(\Omega^{p}_{\tX}\log(E)\otimes \cO_{\tX}(D^{> 0}_{i}))\otimes \cL^{-i}.$$
    $$\pi_{*}\cD(\underline{\Omega}^{n-p}_{Y}) \simeq \Omega^{p}_{X} \oplus \bigoplus^{N-1}_{i=1} \bR \phi_{*}(\Omega^{p}_{\tX}\log(E)\otimes \cO_{\tX}(D^{ \geq0}_{i}))\otimes \cL^{-i}.$$
   
\end{thm}
\begin{proof}
For the extension  $\varrho_{!}\cV$, the differential map $\cO_{\tY} \rightarrow \Omega^{1}_{\tY}$ induces a logarithmic connection
$$\phi^{*}\cL^{-i} \rightarrow \Omega^{1}_{\tX}(\log E) \otimes \phi^{*}\cL^{-i} \quad \text{for $1 \leq i \leq N-1.$}$$
After a possible twist by $\cO_{\tX}(D^{> 0}_{i})$, with $D^{> 0}_{i}$ constructed by the values \ref{m_{ij}},  the eigenvalues of $res \nabla$ along the irreducible components of $E$ are contained in $(0,1]$ for the logarithmic connection
$$\phi^{*}\cL^{-i}\otimes \cO_{\tX}(D^{> 0}_{i}) \rightarrow \Omega^{1}_{\tX}(\log E) \otimes \phi^{*}\cL^{-i} \otimes \cO_{\tX}(D^{> 0}_{i}). $$
 By Corollary \ref{LogCor}, 
$$ Gr^{F}_{-p}DR(\varrho_{!}\cV) \simeq \displaystyle \bigoplus_{i=1}^{N-1}\Omega^{p}_{\tX}(\log E) \otimes \phi^{*}\cL^{-i} \otimes \cO_{\tX}(D^{>0}_{i})[n-p].$$
Therefore, because the induced map $\phi: \tX \backslash E \rightarrow X \backslash D$ is an isomorphism, there are quasi-isomorphisms
$$Gr^{F}_{-p}DR(j_{!}\cV) \simeq  \bR \phi_{*} Gr^{F}_{-p}DR(\varrho_{!}\cV) \simeq \displaystyle \bigoplus_{i=1}^{N-1} \bR \phi_{*}(\Omega^{p}_{\tX}\log(E)\otimes \cO_{\tX}(D^{> 0}_{i}))\otimes \cL^{-i}[n-p].$$
Hence, we obtain the quasi-isomorphism
$$\pi_{*}\OmpY \simeq \Omega^{p}_{X} \oplus \bigoplus^{N-1}_{i=1} \bR \phi_{*}(\Omega^{p}_{\tX}\log(E)\otimes \cO_{\tX}(D^{>0}_{i}))\otimes \cL^{-i}.$$
We also have the decomposition
$$\pi_{+}\D(\Q^{H}_{Y}[n])(-n) \simeq \Q^{H}_{X}[n] \oplus j_{+}\cV.$$
Similarly, for $\tX$, we have 
$$\tilde{\pi}_{+}(\D(\Q^{H}_{\tY}[n])(-n)) \simeq \Q^{H}_{\tX}[n] \oplus \varrho_{+}\cV.$$
For the extension  $\varrho_{+}\cV$, after a possible twist by $\cO_{\tX}(D^{ \geq 0}_{i})$, with $D^{\geq 0}_{i}$ constructed by the values \ref{m'_{ij}}, the eigenvalues of $res \nabla$ along the irreducible components of $E$ are contained in $(0,1]$ for the logarithmic connection
$$\phi^{*}\cL^{-i}\otimes \cO_{\tX}(D^{\geq 0}_{i}) \rightarrow \Omega^{1}_{\tX}(\log E) \otimes \phi^{*}\cL^{-i} \otimes \cO_{\tX}(D^{\geq 0}_{i})$$
By Corollary \ref{LogCor}, there are quasi-isomorphisms
$$Gr^{F}_{-p}DR(j_{+}\cV) \simeq \bR \phi_{*} Gr^{F}_{-p}DR(\varrho_{+}\cV) \simeq \displaystyle \bigoplus_{i=1}^{N-1} \bR \phi_{*}(\Omega^{p}_{\tX}\log(E)\otimes \cO_{\tX}(D^{\geq 0}_{i}))\otimes \cL^{-i}[n-p].$$
Hence, we obtain the quasi-isomorphism
$$\pi_{*}\cD(\OmpY) \simeq \Omega^{n-p}_{X} \oplus \bigoplus^{N-1}_{i=1} \bR \phi_{*}(\Omega^{n-p}_{\tX}\log(E)\otimes \cO_{\tX}(D^{\geq 0}_{i}))\otimes \cL^{-i}.$$
\end{proof}

\begin{rmk}
If $\phi:\tX \rightarrow X$ is a strong log resolution of the pair $(X,D)$ given in the theorem above, with  $\phi^{*}(D)= \displaystyle \sum a_{j}E_{j},$ where $E_{j}$ are the irreducible components of the divisor $E = \phi^{*}D$, then the divisors $D^{> 0}_{i}$ and $D^{ \geq0}_{i}$ are constructed in the following way, 
$$D^{\geq 0}_{i} = \displaystyle \sum \bigg\lfloor \frac{ia_{j}}{N} \bigg \rfloor E_{j}:= \displaystyle
 \bigg\lfloor \frac{i}{N}\phi^{*}D \bigg\rfloor$$
 
 $$D^{> 0}_{i} = \displaystyle \sum b_{j}E_{j} \quad  \text{where   $b_{j} =$}\begin{cases}
    \displaystyle\bigg\lfloor \frac{ia_{j}}{N} \bigg \rfloor & \text{ if  $\displaystyle\frac{ia_{j}}{N} \notin \Z$}\\ \\
     \displaystyle \frac{ia_{j}}{N} -1 & \text{ if  $\displaystyle \frac{ia_{j}}{N} \in \Z$}
 \end{cases}$$
 \end{rmk}

 \begin{rmk}\label{welldefined}
The complexes $\bR \phi_{*}(\Omega^{p}_{\tX}\log(E)\otimes \cO_{\tX}(D^{> 0}_{i}))$ and $\bR \phi_{*}(\Omega^{p}_{\tX}\log(E)\otimes \cO_{\tX}(D^{\geq 0}_{i}))$ constructed in Theorem \ref{SmoothDecomp} are independent of the strong log resolution (up to quasi-isomorphism) because the complexes $Gr^{F}_{-p}DR(j_{!}\cV)$ and $Gr^{F}_{-p}DR(j_{+}\cV)$ are independent of the strong log resolution. As filtered $D_{U}$-modules, there is a splitting $\cV \cong \displaystyle \bigoplus_{i=1}^{N-1}\cH_{i}$ and there are quasi-isomorphims
$$Gr^{F}_{-p}DR(j_{!}\cH_{i})[p-n] \simeq \bR \phi_{*}(\Omega^{p}_{\tX}\log(E)\otimes \cO_{\tX}(D^{> 0}_{i})) \otimes \cL^{-i}$$
$$Gr^{F}_{-p}DR(j_{+}\cH_{i})[p-n] \simeq \bR \phi_{*}(\Omega^{p}_{\tX}\log(E)\otimes \cO_{\tX}(D^{\geq 0}_{i})) \otimes \cL^{-i}$$
 \end{rmk}

 \begin{cor}\label{Vanishing}
            Let $\pi:Y \rightarrow X$ be this $N$-fold cyclic covering of $D$. If $\phi:\tX \rightarrow X$ is a strong log resolution of the pair $(X,D)$ with divisors  $D^{> 0}_{i}$ and $D^{ \geq0}_{i}$ on $\tX$ given above, then for $1 \leq i \leq N-1$ and $p+q>n,$
            $$R^{q}\phi_{*}(\Omega^{p}_{\tX}\log(E) \otimes \cO_{\tX}(D^{> 0}_{i})) = 0 \quad \text{and} \quad R^{q}\phi_{*}(\Omega^{p}_{\tX}\log(E) \otimes \cO_{\tX}(D^{\geq 0}_{i})) = 0.$$
         Alternatively, for $1 \leq i \leq N-1$ and $p+q>n,$
            $$R^{q}\phi_{*}(\Omega^{p}_{\tX}\log(E)(-E) \otimes \cO_{\tX}(-D^{> 0}_{i})) = 0 \quad \text{and} \quad R^{q}\phi_{*}(\Omega^{p}_{\tX}\log(E)(-E) \otimes \cO_{\tX}(-D^{\geq 0}_{i})) = 0.$$
            In particular, for $1 \leq i \leq N-1$ and $q >0,$
            $$R^{q}\phi_{*}(\omega_{\tX}\otimes \cO_{\tX}(-D^{> 0}_{i})) = 0 \quad \text{and} \quad R^{q}\phi_{*}(\omega_{\tX} \otimes \cO_{\tX}(-D^{\geq 0}_{i})) = 0.$$

            \end{cor}

            \begin{proof}
               We prove the statement for the divisors $D^{>0}_{i};$ the case for the divisors $D^{\geq0}_{i}$ is done in a similar fashion.
               
            From the previous theorem,
            $$Gr^{F}_{-p}DR(j_{!}\cV) \simeq \displaystyle \bigoplus_{i=1}^{N-1} \bR \phi_{*}(\Omega^{p}_{\tX}\log(E)\otimes \cO_{\tX}(D^{> 0}_{i}))\otimes \cL^{-i}[n-p].$$
            Since $j_{!}\cV$ is a mixed Hodge module on $X$, we have $\cH^{i}(Gr^{F}_{-p}DR(j_{!}\cV)) = 0$ for $i >0.$ From the quasi-isomorphism above, we have the first vanishing result.

            There are also  quasi-isomorphisms
    $$Gr^{F}_{p-n}DR(\varrho_{!}\cV) \simeq \displaystyle \bigoplus_{i=1}^{N-1}  \Omega^{n-p}_{\tX}\log(E) \otimes\cO_{\tX}(D^{> 0}_{i})\otimes \phi^{*}\cL^{-i}[p],$$
    $$\bR \cH om_{\cO_{\tX}}\bigg(\displaystyle \bigoplus_{i=1}^{N-1}\Omega^{n-p}_{\tX}\log(E) \otimes\cO_{\tX}(D^{> 0}_{i})\otimes \phi^{*}\cL^{-i}, \omega_{\tX}\bigg) =\cD(Gr^{F}_{p-n}DR(\varrho_{!}\cV))[p] \simeq Gr^{F}_{-p}DR(\varrho_{+}\cV)[p-n].$$
    Therefore,
    $$\displaystyle \bigoplus_{i=1}^{N-1} R^{q}\phi_{*}(\Omega^{p}_{\tX}\log(E)(-E)\otimes  \cO_{\tX}(-D^{> 0}_{i}))\otimes \cL^{i} \cong R^{q+p-n}\phi_{*}Gr^{F}_{-p}DR(\varrho_{+}\cV) $$
    $$\cong \cH^{q+p-n}\bigg(Gr^{F}_{-n}DR(\phi_{+}\varrho_{+}\cV)\bigg) \cong \cH^{q+p-n}\bigg(Gr^{F}_{-n}DR(j_{+}\cV)\bigg).$$
    Since $j_{+}\cV$ is a mixed Hodge module on $X$, $\cH^{q+p-n}\bigg(Gr^{F}_{-n}DR(j_{+}\cV)\bigg)=0$ for $q+p-n>0.$ This gives us the second vanishing result.
            \end{proof}

            \begin{rmk}
                A similar vanishing result was also given by Chen and Musta\c{t}\u{a} \cite{chenMustata}.
            \end{rmk}

\begin{cor}\label{singCor}
    With the same notation as the previous theorem, we have
    \begin{enumerate}
        \item $Y$ has Du Bois singularities if and only if for $1 \leq i \leq N-1$ the natural map $\cO_{X} \rightarrow  \bR \phi_{*}\cO_{\tX}(D^{>0}_{i})$ is a quasi-isomorphism.\\
        
        \item $Y$ has rational singularities if and only if for $1 \leq i \leq N-1$ the natural map $\cO_{X} \rightarrow \bR \phi_{*}\cO_{\tX}(D^{\geq 0}_{i})$ is a quasi-isomorphism. \\

        \item If $\cO_{\tX}(D^{>0}_i) = \cO_{\tX}(D^{\geq 0}_{i})$ for $1 \leq i \leq N-1$, then $\Q^{H}_{Y}[n] \cong IC^{H}_{Y}.$
    \end{enumerate}
\end{cor}

\begin{proof}
\begin{enumerate}
 \item If $Y \rightarrow X$ is this $N$-fold cyclic covering of $D$, then we have a decomposition
 $$\pi_{*}\cO_{Y} \cong \cO_{X} \oplus \bigoplus_{i=1}^{N-1}\cL^{-i}.$$
 The variety $Y$ has Du Bois singularities if and only if the natural map $\cO_{Y} \rightarrow \underline{\Omega}^{0}_{Y} $ is a quasi-isomorphism. Since $\pi:Y \rightarrow X$ is affine, $Y$ has Du Bois singularities if and only if the natural map
 $$\pi_{*}\cO_{Y}  \rightarrow \pi_{*}\underline{\Omega}^{0}_{Y} $$
 is a quasi-isomorphism. By the previous theorem, $Y$ has Du Bois singularities if and only if the natural map
 $$ \cO_{X} \oplus \bigoplus_{i=1}^{N-1}\cL^{-i} \cong \pi_{*}\cO_{Y} \rightarrow \pi_{*}\underline{\Omega}^{0}_{Y} \simeq \cO_{X} \oplus \bigoplus_{i=1}^{N-1} \bR \phi_{*}\cO_{\tX}(D^{>0}_{i}) \otimes \cL^{-i}$$
 is a quasi-isomorphism.
 Thus, $Y$ has Du Bois singularities if and only if the natural map $ \cO_{X} \rightarrow \bR \phi_{*}\cO_{\tX}(D^{>0}_{i})$ is a quasi-isomorphism for $1 \leq i \leq N-1.$

 \item If $f:Y' \rightarrow Y$ is a resolution of singularities, then $\cD(\underline{\Omega}^{n}_{Y}) \simeq \bR f_{*}\cO_{Y'}.$ Therefore, $Y$ has rational singularities if and only if $$ \cO_{X} \oplus \bigoplus_{i=1}^{N-1}\cL^{-i} \cong \pi_{*}\cO_{Y} \rightarrow \pi_{*}\cD(\underline{\Omega}^{n}_{Y}) $$
 is a quasi-isomorphism. Using a similar argument as the previous case, $Y$ has rational singularities if and only if  the natural map $ \cO_{X} \rightarrow \bR \phi_{*}\cO_{\tX}(D^{\geq0}_{i})$ is a quasi-isomorphism for $1 \leq i \leq N-1.$

 \item  Since $X$ is smooth, we have the isomorphisms
$$\pi_{+}IC^{H}_{Y} \cong  \Q^{H}_{X}[n] \oplus IC_{X}(\cV)  \cong  \Q^{H}_{X}[n] \oplus  \im [j_{!}\cV \rightarrow j_{+}\cV].$$
If $\cO_{\tX}(D^{>0}_{i}) = \cO_{\tX}(D^{\geq 0}_{i})$, then the map $j_{!}\cV \rightarrow j_{+}\cV$ is an isomorphism. Thus $\pi_{+}IC^{H}_{Y} \cong \pi_{+}\Q^{H}_{Y}[n]$, which implies $IC^{H}_{Y} \cong \Q^{H}_{Y}[n]$.

 \end{enumerate}

\end{proof}
\begin{ex}
    
    Consider the $2$-fold cyclic covering 
    $$Y = \Spec \bigg( \C[x, y, t]/\langle t^{2}-(x^3 - y^6)\rangle \bigg)\rightarrow \Spec\bigg(\C[x,y]\bigg)=X$$
    with $D = \{x^3-y^6 = 0\}$. Let $\phi_{1}:X_{1} \rightarrow X$ be the blow-up of the origin. After a second blow-up, $$\phi:\tX=X_{2} \xrightarrow{\phi_{2}} X_{1} \xrightarrow{\phi_{1}} X,$$
    $$\phi^{*}(\cO_{X}(D)) \cong \cO_{X_{2}}(3E_{1}) \otimes \cO_{X_{2}}(6E_{2})\otimes \cO_{X_{2}}(C_{1}) \otimes \cO_{X_{2}}(C_{2}) \otimes \cO_{X_{2}}(C_{3})$$
    where $E_{2} = \mathbb{P}_{1}$ is the exceptional divisor of $\phi_{2}$, $C_{i}$ are the smooth irreducible components of the strict transform of $D$, and $E_{1}$ is the strick transform of the exceptional divisor of $\phi_{1}.$ The divisor $E=3E_{1} + 6E_{2} + C_{1} + C_{2} + C_{3}$ has simple normal crossing support, and the induced map $\phi: \tX\backslash E \rightarrow X \backslash D$ is an isomorphism. For the logarithmic connection
    $$\phi^{*}\cL^{-1} \rightarrow \Omega^{1}_{\tX}(\log E) \otimes \cL^{-1}$$
    to have eigenvalues of $res \nabla$ along the irreducible components  to be between $(0,1]$, we need to twist the connection by
    $$\cO_{\tX}(D^{>0}) = \cO_{\tX}(E_{1}) \otimes \cO_{\tX}(2E_{2}).$$
    For the logarithmic connection to have eigenvalues of $res \nabla$ along the irreducible components between $[0,1)$, we need to twist the connection by the divisor
    $$\cO_{\tX}(D^{\geq 0})= \cO_{\tX}(E_{1}) \otimes \cO_{\tX}(3E_{2}).$$
    There are quasi-isomorphisms
    $$\pi_{*}\OmpY  \simeq  \Omega^{p}_{X} \oplus \bR \phi_{*}(\Omega^{p}_{\tX}\log(E)\otimes \cO_{\tX}(D^{>0}))\otimes \cL^{-1} \quad \text{for $0 \leq p \leq 2$.}$$
    
    $$ \pi_{*}\cD(\underline{\Omega}^{n-p}_{Y}) \simeq  \Omega^{p}_{X} \oplus \bR \phi_{*}(\Omega^{p}_{\tX}\log(E)\otimes \cO_{\tX}(D^{\geq 0}))\otimes \cL^{-1} \quad \text{for $0 \leq p \leq 2$.}$$
If $\omega_{\tX}$ and $\omega_{X}$ are the canonical bundles for $\tX$ and $X$, respectively, then 
    $$\omega_{\tX} \cong \phi^{*}\omega_{X} \otimes \cO_{\tX}(E_{1}) \otimes \cO_{\tX}(2E_{2}) \cong \phi^{*}\omega_{X} \otimes \cO_{\tX}(D^{>0}).$$
   We have
        $$\bR \phi_{*}(\cO_{\tX}(D^{>0})) = \bR \phi_{*}\bigg( \omega_{\tX} \otimes \phi^{*}\omega^{-1}_{X}\bigg) \simeq \bR \phi_{*}\omega_{\tX} \otimes \omega^{-1}_{X} \simeq \cO_{X}.$$
        By the previous corollary, $Y$ has Du Bois singularities. However, $Y$ does not have rational singularities. Indeed, we have
$$\cO_{\tX}(D^{\geq 0}) = \cO_{\tX}(D^{>0}) \otimes \cO_{X_{2}}(E_{2}).$$
Consider the following isomorphism,
$$\omega_{\tX} \otimes \phi^{*}\omega^{-1}_{X} \otimes \cO_{\tX}(E_{2}) \cong  \cO_{\tX}(D^{\geq 0}).$$
Also, consider the short exact sequence
$$0 \rightarrow \cO_{\tX}(-E_{2}) \rightarrow \cO_{\tX} \rightarrow \cO_{E_{2}} \rightarrow 0.$$
If we tensor the previous short exact sequence by $\cO_{\tX}(D^{\geq 0})$, we obtain
$$0 \rightarrow \omega_{\tX} \otimes \phi^{*}\omega^{-1}_{X} \rightarrow \cO_{\tX}(D^{\geq 0}) \rightarrow \omega_{E_{2}}\otimes \phi^{*}\omega^{-1}_{X}\rightarrow 0,$$
where $\omega_{E_{2}}$ is the canonical bundle on $E_{2}.$ If we apply $\bR \phi_{*}$ to this short exact, we obtain the exact triangle
$$\bT  \cO_{X} \arrow[r] & \bR \phi_{*}\cO_{\tX}(D^{\geq 0}) \arrow[r] & \bR \phi_{*}\omega_{E_{2}}\otimes \omega^{-1}_{X} \arrow[r,"+1"] & \hfill \eT$$
Taking cohomology, we obtain the long exact sequence
$$0 \rightarrow \cO_{X} \rightarrow \phi_{*}\cO_{\tX}(D^{\geq 0}) \rightarrow \phi_{*}\omega_{E_{2}}\otimes \omega^{-1}_{X} \rightarrow 0 \rightarrow R^{1}\phi_{*}\cO_{\tX}(D^{\geq 0}) \rightarrow R^{1}\phi_{*}\omega_{E_{2}}\otimes \omega^{-1}_{X} \rightarrow 0.$$
Since $E_{2} = \mathbb{P}^1$,  $\phi_{*}\omega_{E_{2}} = 0.$ The isomorphism $\cO_{X} \cong \phi_{*}\cO_{\tX}(D^{\geq 0})$ implies $Y$ is normal. We also have $R^{1}\phi_{*}\omega_{E_{2}} \neq 0$. The isomorphism $R^{1}\phi_{*}\cO_{\tX}(D^{\geq 0}) \cong  R^{1}\phi_{*}\omega_{E_{2}}\otimes\omega^{-1}_{X} \neq 0$ implies $Y$ does not have rational singularities.
\end{ex}

\begin{rmk}
    It is well known that the variety $Y = \Spec \bigg( \C[x, y, t]/\langle t^{2}-(x^3 - y^6)\rangle \bigg)$ has Du Bois singularities and does not have rational singularities because the minimal exponent of the quasi-homogeneous polynomial $t^2 - (x^3 - y^6)$ is given by 
    $$\displaystyle \frac{1}{2}+ \frac{1}{3} + \frac{1}{6}=1 \quad \text{\cite{Saito-DB}.}$$
    
\end{rmk}

\begin{ex}
     For this example, consider the $7$-fold cyclic covering 
    $$Y = \Spec \bigg( \C[x, y, t]/\langle t^{7}-(x^2 - y^3)\rangle \bigg)\rightarrow \Spec\bigg(\C[x,y]\bigg)=X$$
    with $D = \{x^2-y^3 = 0\}$. It is well known that if we take three consecutive blow-ups,
    $$\phi: \tX =  X_{3} \xrightarrow{\phi_{3}} X_{2} \xrightarrow{\phi_{2}} X_{1} \xrightarrow{\phi_{1}} X$$
    the induced map $\phi: \tX\backslash E \rightarrow X \backslash D$ is an isomorphism, where $E = \phi^{*}D =2E_{1} + 3E_{2} + 6E_{3} +C$ has simple normal crossing support. The divisors $E_{i}$ are the exceptional divisors, and $C$ is the strict transform of $D.$ 
     For the logarithmic connection
    $$\phi^{*}\cL^{-i} \rightarrow \Omega^{1}_{\tX}(\log E) \otimes \cL^{-i} $$
    to have eigenvalues of $res \nabla$ along the irreducible components  between $(0,1]$, or between $[0,1)$, we need to twist the connection by the following respective line bundles
    $$\cO_{\tX}(D^{\geq 0}_{1}) = \cO_{\tX}(D^{>0}_{1}) =\cO_{\tX}$$
    $$\cO_{\tX}(D^{\geq 0}_{2}) = \cO_{\tX}(D^{>0}_{2}) =\cO_{\tX}(E_{3})$$
    $$\cO_{\tX}(D^{\geq 0}_{3}) = \cO_{\tX}(D^{>0}_{3}) =\cO_{\tX}(E_{2}) \otimes \cO_{\tX}(2E_{3})$$
    $$\cO_{\tX}(D^{\geq0}_{4}) = \cO_{\tX}(D^{>0}_{4}) =\cO_{\tX}(E_{1}) \otimes \cO_{\tX}(E_{2}) \otimes \cO_{\tX}(3E_{3})$$
    $$\cO_{\tX}(D^{\geq0}_{5}) = \cO_{\tX}(D^{>0}_{5}) =\cO_{\tX}(E_{1}) \otimes \cO_{\tX}(2E_{2}) \otimes \cO_{\tX}(4E_{3})$$
    $$ \cO_{\tX}(D^{\geq 0}_{6}) = \cO_{\tX}(D^{>0}_{6}) =\cO_{\tX}(E_{1}) \otimes \cO_{\tX}(2E_{2}) \otimes \cO_{\tX}(5E_{3}).$$
   If $\omega_{\tX}$ and $\omega_{X}$ are the canonical bundles for $\tX$ and $X$, respectively, then 
    $$\omega_{\tX} \cong \phi^{*}\omega_{X} \otimes \cO_{\tX}(E_{1}) \otimes \cO_{\tX}(2E_{2}) \otimes \cO_{\tX}(4E_{3}).$$
    It can be shown that $\bR \phi_{*}\cO_{\tX}(D^{>0}_{i}) \simeq \bR \phi_{*}\cO_{\tX}(D^{\geq0}_{i}) \simeq \cO_{X}$ for $1 \leq i \leq 5.$ However, using the same argument as the previous example,
    $$\phi_{*}\cO_{\tX}(D^{>0}_{6}) \cong \phi_{*}\cO_{\tX}(D^{\geq0}_{6}) \cong \cO_{X} \quad \text{and} \quad R^{1}\phi_{*}\cO_{\tX}(D^{>0}_{6}) \neq 0.$$
    Thus, $Y$ is normal, but does not have Du Bois singularities.
    Also, since $\cO_{\tX}(D^{\geq 0}_{i}) =\cO_{\tX}(D^{>0}_{i})$ for $1 \leq i \leq 6,$ we have $\Q^{H}_{Y}[n] \cong IC^{H}_{Y},$ meaning that $Y$ is a rational homology manifold.
\end{ex}

 \begin{rmk}
    It was first shown by Steenbrink \cite{SteenbrinkIsolated} that the variety $Y$ does not have Du Bois singularities. Also, the variety $Y = \Spec \bigg( \C[x, y, t]/\langle t^{7}-(x^2 - y^3)\rangle \bigg)$ does not have Du Bois singularities because the minimal exponent of the quasi-homogeneous polynomial $t^7 - (x^2 - y^3)$ is given by
    $$\displaystyle \frac{1}{2}+ \frac{1}{3} + \frac{1}{7}<1 \quad \text{\cite{Saito-DB}.}$$
    Since $Y$ is a rational homology manifold, we have quasi-isomorphisms
    $$\OmpY \simeq I\OmpY:= Gr^{F}_{-p}DR(IC^{H}_{Y})[p-n] \quad \text{for $0 \leq p \leq 2.$}$$

\end{rmk}

\subsection{On The Log Canonical Threshold Of Hypersurfaces}

       Let $\phi: \tX \rightarrow X$ be a strong log resolution of the pair $(X, D).$ We will have $E = \phi^{*}D_{red}$ denote the simple normal crossing divisor on $\tX$ and we will set 
        $$\phi^{*}D= \displaystyle \sum a_{j}E_{j}.$$
        For $\lambda \in \Q$, the \emph{multiplier ideal sheaf}
        $$\mathcal{J}(\lambda \cdot D)= \mathcal{J}(X,\lambda \cdot D) \subseteq \cO_{X}$$
        associated to the $\Q$-divisor $\lambda \cdot D$ is defined to be
        $$\mathcal{J}(\lambda \cdot D)= \phi_{*}\cO_{\tX}(K_{\tX/X} - \lfloor \lambda \cdot \phi^{*}D \rfloor).$$
        For a closed point $x \in X$, the \emph{log canonical threshold} of the divisor $D$ at $x$ is defined by
        $$\lct(D,x):= \inf \bigg \{ \lambda\in \Q \hspace{.05in}| \hspace{.05in} \mathcal{J}(X, \lambda \cdot D)_{x} \subseteq m_{x} \bigg\},$$
        where $m_{x} \subset \cO_{X}$ is the maximal ideal sheaf of $x.$ In general, the log canonical threshold of $D$ is defined as
        $$\lct(D):= \inf \bigg \{ \lct(D,x) \hspace{.05in}| \hspace{.05in} x \in X\bigg \}$$
        If $K_{\tX}$ and $K_{X}$ are the canonical divisors on $\tX$ and $X$ respectively, then we can write
        $$K_{\tX / X} = \displaystyle \sum k_{j}E_{j},$$
        and it is also well known that we can calculate the log canonical threshold of $D$ with the coefficients of $\phi^{*}D$ and $K_{\tX/X},$
        $$\lct(D) = \min \bigg \{1, \bigg \{ \displaystyle \frac{k_{j}+1}{a_{j}} \hspace{.05in}\bigg| \hspace{.05in} \text{$E_{i}$ are exceptional divisors} \bigg \}  \bigg\}.$$
        Analytically, the \emph{log canonical threshold} of the local equation $f$, also denoted as $\lct(f),$ is defined by
        $$\lct(f):= \sup \bigg\{\lambda \in \Q \hspace{.05in} \bigg| \hspace{.05in} \displaystyle \frac{1}{|f|^{2\lambda}} \in L^1_{loc} \bigg\}.$$
        If $D$ can be written as $D = \displaystyle \sum b_{i}D_{i}$, where $D_{i}$ are the irreducible components of $D$, then there are local equations $g_{i}$ defining the hypersurfaces $D_{i}$ and the multiplier ideal sheaf of $D$ is viewed as the ideal sheaf
        $$\displaystyle \mathfrak{J}(\varphi_{D}) \underset{locally}{=} \bigg \{ h \in \cO^{an}_{X} \bigg| \displaystyle \frac{|h|^{2}}{ \prod |g_{i}|^{2b_{i}}} \in L^{1}_{loc} \bigg\},$$
        where $\varphi_{D} = \displaystyle \sum b_{i}\log|g_{i}|$ is a locally defined plurisubharmonic function. We have 
        $$\mathcal{J}(D)^{an} = \mathfrak{J}(\varphi_{D}),$$
        and the two distinct views of the local canonical threshold of a divisor, algebraically or analytically, coincide. For a more detailed and in-depth discussion of the multiplier ideal sheaf and the log canonical threshold, see \cite{LarasfeldII}.
        
        Alternatively, let $X$ be a smooth affine algebraic variety and $f \in \cO_{X}(X)$. There exists a polynomial $b(s) \in \C[s]$, and a polynomial $P(s) \in \cD_{X}[s]$ satisfying the relation
        $$P(s)f^{s+1} = b(s) \cdot f^{s}.$$
        The set of all polynomials $b(s)$ forms an ideal in the polynomial ring $\C[s]$, and the monic generator of this ideal, denoted as $b_{f}(s),$ is called the \emph{Bernstein-Sato polynomial} or the $\emph{b-function}$ of $f$. The roots of the polynomial $b_{f}(-s)$ are positive rational numbers \cite{Kash-bfcn}, and the minimal root of  $b_{f}(-s)$ is the log canonical threshold $\lct(f)$ \cite[Thm. 10.6]{KollarPairs}.   The minimal root of the reduced Bernstein-Sato polynomial $ \tilde{b}_{f}(-s):=\displaystyle \frac{b_{f}(-s)}{s-1}$, which is denoted as $\tilde{\alpha}_{f}$, is called the \emph{minimal exponent of $f$}. The relationship between the minimal exponent of the local defining equation $f$ and the log canonical threshold of $f$ is given by
        $$\lct(f) := \min \bigg \{ \tilde{\alpha}_{f}, 1 \bigg \}.$$

      Next, we show that we can precisely determine when the natural maps 
      $$ \cO_{X} \rightarrow \bR \phi_{*}(\cO_{\tX}(D^{>0}_{i}))  \quad \text{and} \quad \cO_{X} \rightarrow \bR \phi_{*}(\cO_{\tX}(D^{\geq 0}_{i})),$$
            are quasi-isomorphisms in terms of the log canonical threshold of $D.$
       \begin{thm}\label{lctThm}
           If  $Y \rightarrow X$ is this $N$-fold cyclic covering of $D$ and $\phi:\tX \rightarrow X$ is a strong log resolution of the pair $(X,D)$,  
       $$\text{The natural map $\cO_{X} \rightarrow \bR \phi_{*}( \cO_{\tX}(D^{> 0}_{i}))$  is a quasi-isomorphism} \quad \text{if and only if} \quad \lct(D)\geq \dfrac{i}{N}$$
        $$\text{The natural map $\cO_{X} \rightarrow \bR \phi_{*}( \cO_{\tX}(D^{\geq 0}_{i}))$  is a quasi-isomorphism} \quad \text{if and only if} \quad \lct(D)> \dfrac{i}{N}$$
       \end{thm}

        \begin{proof}
             We will set the notation of
            $$\phi^{*}(D) = \displaystyle \sum a_{j}E_{j}$$
            $$K_{\tX/X}= \displaystyle \sum k_{j}E_{j},$$
            where $E_{j}$ are the irreducible components of the simple normal crossing divisor $E.$ We claim the natural map $\cO_{X} \rightarrow \bR \phi_{*}( \cO_{\tX}(D^{> 0}_{i}))$  is a quasi-isomorphism, respectively the natural map $\cO_{X} \rightarrow \bR \phi_{*}( \cO_{\tX}(D^{\geq 0}_{i}))$  is a quasi-isomorphism, if and only if
      $$K_{\tX/X} - D^{>0}_{i} \quad \text{ is effective}$$
      respectively,
       $$K_{\tX/X} - D^{\geq0}_{i} \quad \text{ is effective}.$$
       We will prove the claim for the natural map $\cO_{X} \rightarrow \bR \phi_{*}( \cO_{\tX}(D^{> 0}_{i}))$; the argument for $\cO_{X} \rightarrow \bR \phi_{*}( \cO_{\tX}(D^{> 0}_{i}))$ is shown similarly. We set
       $$K_{\tX/X} - D^{>0}_{i} =\displaystyle \sum b_{ij}E_{j} \quad \text{ where $b_{ij} = k_{j} - m_{ij}$ (see \ref{m_{ij}} for $m_{ij}$)}$$
       First, assume that $b_{ij} \geq 0$  for all $j$. By the previous Corollary \ref{Vanishing},
       $$R^{q}\phi_{*}\cO_{\tX}\bigg(\sum b_{ij}E_{j}\bigg) \cong R^{q}\phi_{*}\bigg(\omega_{\tX} \otimes \cO_{\tX}(-D^{>0}_{i})\bigg)\otimes \omega_{X} = 0 \quad \text{for $q >0$.}$$
       We also have $0 \leq b_{ij} \leq k_{j}$ for all $j.$ So, there are inclusions
       $$ \cO_{\tX} \subseteq \cO_{\tX}\bigg(\sum b_{ij}E_{j}\bigg) \subseteq \omega_{\tX} \otimes \phi^{*}\omega_{X} $$
       So, we have $\cO_{X} \subseteq \phi_{*}\cO_{\tX}\bigg(\sum b_{ij}E_{j}\bigg) \subseteq \cO_{X}$, which implies $\phi_{*}\cO_{\tX}\bigg(\sum b_{ij}E_{j}\bigg) \cong \cO_{X}.$ So, altogether, there is a quasi-isomorphism
       $$\bR \phi_{*} \cO_{\tX}\bigg(\sum b_{ij}E_{j}\bigg) \simeq \cO_{X}.$$
       This implies we have the following quasi-isomorphisms,
       $$\cO_{X} \simeq \bR \phi_{*} \cO_{\tX}\bigg(\sum b_{ij}E_{j}\bigg) \simeq \bR \phi_{*}\bigg( \omega_{\tX} \otimes \phi^{*}\omega^{-1}_{X} \otimes \cO_{\tX}(-D^{>0}_{i}) \bigg)$$
       $$\simeq \bR \phi_{*} \bR \cH om_{\cO_{\tX}}(\cO_{\tX}(D^{>0}_{i}), \phi^{!}\cO_{X}) \simeq \bR \cH om_{\cO_{X}}(\bR \phi_{*}\cO_{\tX}(D^{>0}_{i}),\cO_{X}).$$
       Dualizing, we obtain 
       $$\cO_{X} \simeq \bR \cH om_{\cO_{X}}(\cO_{X},\cO_{X}) \simeq \bR \cH om_{\cO_{X}}(\bR \cH om_{\cO_{\tX}}(\bR \phi_{*}\cO_{\tX}(D^{>0}_{i}),\cO_{X}),\cO_{X})\simeq \bR \phi_{*}\cO_{\tX}(D^{>0}_{i}). $$

        Now assume that $b_{ij} < 0$  for some $j$. For notation, say $b_{i \beta} <0.$ Then there must exist $q>0$ such that 
        $$R^{q}\phi_{*}(\cO_{\tX}(D^{>0}_{i})) \neq 0.$$
        Indeed, suppose on the contrary that $R^{q}\phi_{*}(\cO_{\tX}(D^{>0}_{i})) =0$ for $q>0.$ Consider the dual of $\bR \phi_{*}(\cO_{\tX}(D^{>0}_{i})),$
        $$\bR \cH om_{\cO_{X}}(\bR \phi_{*}(\cO_{\tX}(D^{>0}_{i})), \omega_{X}) \simeq \bR \phi_{*}\bR \cH om_{\cO_{\tX}}((\cO_{\tX}(D^{>0}_{i})), \omega_{\tX})$$
       $$\simeq \bR \phi_{*} \bigg(\cO_{\tX}\bigg(\sum b_{i j}E_{j}\bigg) \bigg) \otimes \omega_{X}^{-1}$$
       Since $E_{j}$ are exceptional divisors, the image $Z:=\phi(E ) \subset X$ has a codimension of at least two. Now, if $R^{q}\phi_{*}(\cO_{\tX}(D^{>0}_{i})) =0$ for $q>0,$ then by \cite[Thm. 1.3]{AH2}, 
       $$\cH^{1}_{Z}\bigg(\bR \phi_{*} \bigg(\displaystyle\cO_{\tX}\bigg(\sum b_{i j}E_{j}\bigg) \bigg)\otimes \omega_{X}^{\vee}  \bigg) = 0,$$
       which implies $\phi_{*} \bigg(\displaystyle\cO_{\tX}\bigg(\sum b_{i j}E_{j}\bigg) \bigg) \cong \cO_{X}.$ However, this is a contradiction because $b_{i \beta}<0$. Therefore, there must exist $q>0$ such that $R^{q}\phi_{*}(\cO_{\tX}(D^{>0}_{i})) \neq 0.$

        From the claim, we the natural map $\cO_{X} \rightarrow \bR \phi_{*}( \cO_{\tX}(D^{> 0}_{i}))$  is a quasi-isomorphism\\
        
        $\text{if and only if \quad   $m_{ij} \leq k_{j}$ for all $j$}$\\

         $\text{ if and only if \quad  $\displaystyle \frac{ia_{j}}{N} \leq k_{j}+1$ for all $j$,}$\\
         
         $\text{ if and only if \quad  $\displaystyle \frac{i}{N} \leq \frac{k_{j}+1}{a_{j}}$ for all $j$,}$\\
         
         $\text{ if and only if \quad  $\displaystyle \frac{i}{N} \leq \lct(D).$}$\\

        \end{proof}

        \begin{rmk} If $\pi: Y \rightarrow X$ is the $N$-fold cyclic covering of $D$, it was also shown recently shown by Chen and  Musta\c{t}\u{a}  \cite{chenMustata} at if $D$ is a reduced divisor on $X$, then $\lct(D) > \dfrac{i}{N}$ if and only if  
        $$R^{q}f_{*}( \cO_{\tX}(D^{\geq 0}_{i})) = R^{q}f_{*}( \cO_{\tX}\bigg(\bigg \lfloor \frac{i}{N}\phi^{*}D \bigg \rfloor\bigg)) = 0 \quad \text{for $q \geq 1.$}$$
        Morevover, it was shown in \cite{chenMustata} that if $D$ is a reduced divisor on $X$ and  $\tilde{\alpha}(D) > k$, then $\tilde{\alpha}(D) > k + \dfrac{i}{N} $ if and only if 
        $$R^{q}f_{*}(\Omega^{k}_{\tX}\log (E) \otimes \cO_{\tX}(D^{\geq 0}_{i})) = 0 \quad \text{for $q \geq 1.$}$$
        \end{rmk}

\section{Cyclic Covers For Singular Varieties}\label{S4}
     For this section,  $X$ will be a singular variety of dimension $n$.  Again, we are in the setting of \ref{setting}.  If 
     $$\pi: Y = \Spec \bigg( \bigoplus_{i=0}^{N-1}\cL^{-i} \bigg) \rightarrow X$$
     is the $N$-fold cyclic covering of $D$, then $$\pi_{+}\Q^{H}_{Y}[n] \simeq \Q^{H}_{X}[n] \oplus j_{!}\cV \quad \text{where $\cV\in D^{b}MHM(U)$.}$$
     Taking the duality functor $\D: D^{b}MHM(X) \rightarrow D^{b}MHM(X)^{op}$, we have
     $$\pi_{+}\D(\Q^{H}_{Y}[n]) \simeq \D(\Q^{H}_{X}[n])  \oplus j_{+}\D(\cV) $$
     \begin{rmk}
       In the case $X$ is singular, we may not have $\Q^{H}_{X}[n]$ or $\cV$ to be pure Hodge modules.
     \end{rmk}

\subsection{Simplicial Resolutions}
Let $\epsilon: X_{\bullet} \rightarrow X$ be a cubical hyperresolution. That is,  $\epsilon: X_{\bullet} \rightarrow X$ is a simiplical resolution and $\dim X_{n-i} \leq i$. We may choose such a simplicial resolution by \cite{GNPP}. We set $X_{-1}= X.$ There are proper maps
$$\epsilon_{ij}:X_{i} \rightarrow X_{i-1} \quad \text{for $0 \leq j \leq i$}$$
which satisfy the conditions
$$\epsilon_{ij}\epsilon_{i+1, j+1} = \epsilon_{ij}\epsilon_{i+1,j} \quad \text{for all $j<i.$}$$
By composing any sequence of $(\epsilon_{ij})$, we obtain a well defined morphism
$$\epsilon^{i}:X_{i} \rightarrow X.$$
Furthermore, by definition, the natural map
$$\Q_{X} \rightarrow \bR \epsilon_{*}\Q_{X_{\bullet}}.$$
is a quasi-isomorphism. By base change, we also have a quasi-isomorphism
$$\C_{X} \rightarrow \bR \epsilon_{*}\C_{X_{\bullet}}.$$
For $i \geq 0$, we have a soft resolution for the constant sheaf $\C_{X_{i}}$
$$0 \rightarrow \C_{X_{i}} \rightarrow \cE^{0}_{X_{i}} \rightarrow \cE^{1}_{X_{i}} \rightarrow \cdots \rightarrow \cE^{\dim X_{i}}_{X_{i}} \rightarrow 0,$$
where $\cE^{k}_{X_{i}}$ is the sheaf of $C^{\infty}$-complex valued $k$-forms on $X_{i}.$ If $k > \dim X_{i}$, we set $\cE^{k}_{X_{i}} = 0$. From the maps $\epsilon_{ij}$ and $\epsilon^{i}$, we have the double complex
$$\bT & 0  & 0  & & 0  &\\
\text{$n^{th}$ row} \quad 0 \arrow[r]& \epsilon^{n}_{*}\cE^{0}_{X_{n}} \arrow[r, "d"] \arrow[u]& \epsilon^{n}_{*}\cE^{1}_{X_{n}} \arrow[r, "d"]\arrow[u] & \cdots \arrow[r, "d"] & \epsilon^{n}_{*}\cE^{n}_{X_{n}} \arrow[r]\arrow[u] & 0 \\
 \text{$(n-1)^{th}$ row} \quad0 \arrow[r] & \epsilon^{n-1}_{*}\cE^{0}_{X_{n-1}} \arrow[r, "d"] \arrow[u] & \epsilon^{n-1}_{*}\cE^{1}_{X_{n-1}} \arrow[r, "d"] \arrow[u] & \cdots \arrow[r, "d"] & \epsilon^{n-1}_{*}\cE^{n}_{X_{n-1}} \arrow[u] \arrow[r] & 0\\
& \vdots \arrow[u] & \vdots \arrow[u] & & \vdots \arrow[u] &\\
\text{$1^{st}$ row} \quad 0 \arrow[r] & \epsilon^{1}_{*}\cE^{0}_{X_{1}} \arrow[r, "d"] \arrow[u, "\epsilon^{*}_{22}- \epsilon^{*}_{21} + \epsilon^{*}_{20}"] & \epsilon^{1}_{*}\cE^{1}_{X_{1}} \arrow[r, "d"] \arrow[u, "\epsilon^{*}_{22}- \epsilon^{*}_{21} + \epsilon^{*}_{20}"] & \cdots \arrow[r, "d"] & \epsilon^{1}_{*}\cE^{n}_{X_{1}} \arrow[u, "\epsilon^{*}_{22}- \epsilon^{*}_{21} + \epsilon^{*}_{20}"] \arrow[r] & 0\\
\text{$0^{th}$ row} \quad 0 \arrow[r] & \epsilon^{0}_{*}\cE^{0}_{X_{0}} \arrow[r, "d"] \arrow[u, "\epsilon^{*}_{11}-\epsilon^{*}_{10}" ] & \epsilon^{0}_{*}\cE^{1}_{X_{0}} \arrow[r, "d"] \arrow[u, "\epsilon^{*}_{11}-\epsilon^{*}_{10}"] & \cdots \arrow[r, "d"] & \epsilon^{0}_{*}\cE^{n}_{X_{0}} \arrow[u, "\epsilon^{*}_{11}-\epsilon^{*}_{10}"] \arrow[r] & 0\\
& 0 \arrow[u] & 0 \arrow[u] & & 0 \arrow[u] &
\eT$$
The $i^{th}$ row of the double complex is quasi-isomorphic to $\bR \epsilon^{i}_{*}\C_{X_{i}}$, and the total complex associated with the double complex above is quasi-isomorphic to $\C_{X}$. Let $C$ denote the double complex above, and $Tot(C)$ denote the total complex. We can filter the total complex by 
$$F^{p}Tot(C)^{m} = 
\displaystyle \bigoplus_{i+j = m, \hspace{0.05in} j \geq p} \epsilon^{j}_{*}\cE^{i}_{X_{j}}$$
With this filtration, we have
$$Gr^{i}_{F}Tot(C) \simeq \begin{cases} \bR \epsilon^{i}_{*}\C_{X_{i}}[-i] & \text{for $0 \leq i \leq n$}\\ \\\
0 & \text{otherwise} \end{cases}$$
\begin{lemma}\label{ConeLemma}
    If $\epsilon: X_{\bullet} \rightarrow X$ is a cubical hyperresolution, then the complex
    $$ Cone(\bR \epsilon^{0}_{*}\C_{X_{0}}, Cone(\bR \epsilon^{1}_{*}\C_{X_{1}}[-1], Cone(\bR \epsilon^{2}_{*}\C_{X_{2}}[-2], Cone(\cdots,Cone(\bR \epsilon^{n-1}_{*}\C_{X_{n-1}}[-(n-1)], \bR \epsilon^{n}_{*}\C_{X_{n}}[-(n-1)])))))$$
    is quasi-isomorphic to $\C_{X}[1].$ 
\end{lemma}

\begin{proof}
    We have the exact sequence
    $$0 \rightarrow F^{1}Tot(C) \rightarrow \C_{X} \rightarrow Gr^{0}_{F}Tot(C) \simeq \bR \epsilon^{0}_{*}\C_{X_{0}} \rightarrow 0.$$
    Using shifts, we have the exact sequence
    $$0 \rightarrow \bR \epsilon^{0}_{*}\C_{X_{0}} \rightarrow F^{1}Tot(C)[1] \rightarrow \C_{X}[1] \rightarrow 0$$
    Thus,
    $$\C_{X}[1] \simeq Cone(\bR \epsilon^{0}_{*}\C_{X_{0}}, F^{1}Tot(C)[1]).$$
    We also have the exact sequence
    $$0 \rightarrow F^{2}Tot(C) \rightarrow F^{1}Tot(C) \rightarrow Gr^{1}_{F}Tot(C) \simeq \bR \epsilon^{1}_{*}\C_{X_{1}}[-1] \rightarrow 0.$$
    Again, using shifts, we have the exact sequence
    $$0  \rightarrow \bR \epsilon^{1}_{*}\C_{X_{1}}[-1] \rightarrow F^{2}Tot(C)[1] \rightarrow F^{1}Tot(C)[1] \rightarrow 0.$$
    Thus, we have the following quasi-isomorphisms
    $$F^{1}Tot(C)[1] \simeq Cone( \epsilon^{1}_{*}\C_{X_{1}}[-1], F^{2}Tot(C)[1]).$$
    $$\C_{X}[1] \simeq Cone(\bR \epsilon^{0}_{*}\C_{X_{0}},Cone( \bR \epsilon^{1}_{*}\C_{X_{1}}[-1], F^{2}Tot(C)[1])).$$
    Continue this process by using the exact triangles
    $$0  \rightarrow \bR \epsilon^{i}_{*}\C_{X_{i}}[-i] \rightarrow F^{i +1}Tot(C)[1] \rightarrow F^{i}Tot(C)[1] \rightarrow 0$$
    and the quasi-isomorphism
    $$Cone(\bR \epsilon^{i}_{*}\C_{X_{i}}[-i],F^{i +1}Tot(C)[1]) \simeq  F^{i}Tot(C)[1].$$
    Since the simplicial resolution $\epsilon: X_{\bullet} \rightarrow X$ is finite, the filtration $F^{\bullet}Tot(C)$ is finite, and this process terminates to the desired quasi-isomorphism.
    
\end{proof}
For $i \geq 0$, we also have a soft resolution for the sheaf of holomorphic $p$-forms $\Omega^{p}_{X_{i}}$
$$0 \rightarrow \Omega^{p}_{X_{i}} \rightarrow \cE^{p,0}_{X_{i}} \rightarrow \cE^{p,1}_{X_{i}} \rightarrow \cdots \rightarrow \cE^{p, \dim X_{i}-p}_{X_{i}} \rightarrow 0,$$
where $\cE^{p,q}_{X_{i}}$ is the sheaf of $C^{\infty}$-complex valued $(p,q)$-forms on $X_{i}.$ From the maps $\epsilon_{ij}$ and $\epsilon^{i}$, we have the double complex
$$\bT & 0  & 0  & & 0  &\\
\text{$n^{th}$ row} \quad 0 \arrow[r]& \epsilon^{n}_{*}\cE^{p,0}_{X_{n}} \arrow[r, "\bar{\partial}"] \arrow[u]& \epsilon^{n}_{*}\cE^{1}_{X_{n}} \arrow[r, "\bar{\partial}"]\arrow[u] & \cdots \arrow[r, "\bar{\partial}"] & \epsilon^{n}_{*}\cE^{p, n-p}_{X_{i}} \arrow[r]\arrow[u] & 0 \\
 \text{$(n-1)^{th}$ row} \quad0 \arrow[r] & \epsilon^{n-1}_{*}\cE^{p,0}_{X_{n-1}} \arrow[r, "\bar{\partial}"] \arrow[u] & \epsilon^{n-1}_{*}\cE^{p,1}_{X_{n-1}} \arrow[r, "\bar{\partial}"] \arrow[u] & \cdots \arrow[r, "\bar{\partial}"] & \epsilon^{n-1}_{*}\cE^{p, n-p}_{X_{n-1}} \arrow[u] \arrow[r] & 0\\
& \vdots \arrow[u] & \vdots \arrow[u] & & \vdots \arrow[u] &\\
\text{$1^{st}$ row} \quad 0 \arrow[r] & \epsilon^{1}_{*}\cE^{p,0}_{X_{1}} \arrow[r, "\bar{\partial}"] \arrow[u, "\epsilon^{*}_{22}- \epsilon^{*}_{21} + \epsilon^{*}_{20}"] & \epsilon^{1}_{*}\cE^{p,1}_{X_{1}} \arrow[r, "\bar{\partial}"] \arrow[u, "\epsilon^{*}_{22}- \epsilon^{*}_{21} + \epsilon^{*}_{20}"] & \cdots \arrow[r, "\bar{\partial}"] & \epsilon^{1}_{*}\cE^{p,n-p}_{X_{1}} \arrow[u, "\epsilon^{*}_{22}- \epsilon^{*}_{21} + \epsilon^{*}_{20}"] \arrow[r] & 0\\
\text{$0^{th}$ row} \quad 0 \arrow[r] & \epsilon^{0}_{*}\cE^{p,0}_{X_{0}} \arrow[r, "\bar{\partial}"] \arrow[u, "\epsilon^{*}_{11}-\epsilon^{*}_{10}" ] & \epsilon^{0}_{*}\cE^{p,1}_{X_{0}} \arrow[r, "\bar{\partial}"] \arrow[u, "\epsilon^{*}_{11}-\epsilon^{*}_{10}"] & \cdots \arrow[r, "\bar{\partial}"] & \epsilon^{0}_{*}\cE^{p,n-p}_{X_{0}} \arrow[u, "\epsilon^{*}_{11}-\epsilon^{*}_{10}"] \arrow[r] & 0\\
& 0 \arrow[u] & 0 \arrow[u] & & 0 \arrow[u] &
\eT$$
The $i^{th}$ row of the double complex is quasi-isomorphic to $\bR \epsilon^{i}_{*}\Omega^{p}_{X_{i}}$. The single complex associated from the double complex is the $p^{th}$-graded piece of the Du Bois complex $\Omp \in D^{b}_{coh}(\cO_{X})$ \cite{dubois}. This construction of the Du Bois complex was given by Steenbrink \cite{steenbrink}. Using a similar argument as above, we may also describe the complex $\Omp$ in the following fashion
$$\Omp[1] \simeq  Cone(\bR \epsilon^{0}_{*}\Omega^{p}_{X_{0}}, Cone(\bR \epsilon^{1}_{*}\Omega^{p}_{X_{1}}[-1],Cone(\cdots,Cone(\bR \epsilon^{n-1}_{*}\Omega^{p}_{X_{n-1}}[-(n-1)], \bR \epsilon^{n}_{*}\Omega^{p}_{X_{n}}[-(n-1)]))))).$$
A similar method of using cones to describe the complex $\Omp$ was given in \cite{shenVenVo}.

\begin{nota}
To help with notation, we set
$$\mathfrak{C}:= Cone(\bullet, \bullet)$$
With this new notation, if $\epsilon:X_{\bullet} \rightarrow X$ is a cubical hyperresolution, we have quasi-isomorphisms
 $$\C[1] \simeq \fC(\bR \epsilon^{0}_{*}\C_{X_{0}}, \fC(\bR \epsilon^{1}_{*}\C_{X_{1}}[-1], \fC(\bR \epsilon^{2}_{*}\C_{X_{2}}[-2], \fC(\cdots,\fC(\bR \epsilon^{n-1}_{*}\C_{X_{n-1}}[-(n-1)], \bR \epsilon^{n}_{*}\C_{X_{n}}[-(n-1)])))))$$
$$\Omp[1] \simeq  \fC(\bR \epsilon^{0}_{*}\Omega^{p}_{X_{0}}, \fC(\bR \epsilon^{1}_{*}\Omega^{p}_{X_{1}}[-1],\fC(\cdots,\fC(\bR \epsilon^{n-1}_{*}\Omega^{p}_{X_{n-1}}[-(n-1)], \bR \epsilon^{n}_{*}\Omega^{p}_{X_{n}}[-(n-1)])))))$$

\end{nota}

\subsection{Simplicial Resolutions And Hodge modules}
Let $\epsilon: X_{\bullet} \rightarrow X$ be a cubical hyperresolution with the assumptions and notation previously discussed. From the proper maps
$$\epsilon_{ij}:X_{i} \rightarrow X_{i-1} \quad \text{for $0 \leq j \leq i$}$$
$$\epsilon^{i}:X_{i} \rightarrow X,$$
for each $i$, we have the following map in the bounded derived category of mixed Hodge modules
$$\displaystyle \sum (-1)^{i+j}(\epsilon_{ij})^{*}:\epsilon^{i-1}_{+}\Q^{H}_{X_{i-1}} \rightarrow \epsilon^{i}_{+}\Q^{H}_{X_{i}}.$$
By construction, the composition map
$$\epsilon^{i-1}_{+}\Q^{H}_{X_{i-1}} \rightarrow \epsilon^{i}_{+}\Q^{H}_{X_{i}} \rightarrow  \epsilon^{i+1}_{+}\Q^{H}_{X_{i+1}}$$
is the zero map in $D^{b}MHM(X)$, which induces a map
$$\epsilon^{i-1}_{+}\Q^{H}_{X_{i-1}} \rightarrow \fC\bigg(\epsilon^{i}_{+}\Q^{H}_{X_{i}}, \epsilon^{i+1}_{+}\Q^{H}_{X_{i+1}} \bigg)[-1] =\fC\bigg(\epsilon^{i}_{+}\Q^{H}_{X_{i}}[-1], \epsilon^{i+1}_{+}\Q^{H}_{X_{i+1}}[-1] \bigg).$$
Now, from the map $\epsilon^{i-2}_{+}\Q^{H}_{X_{i-2}} \rightarrow \epsilon^{i-1}_{+}\Q^{H}_{X_{i-1}}$, there is a map
$$\epsilon^{i-2}_{+}\Q^{H}_{X_{i-2}} \rightarrow \fC\bigg( \epsilon^{i-1}_{+}\Q^{H}_{X_{i-1}}, \fC\bigg(\epsilon^{i}_{+}\Q^{H}_{X_{i}}[-1], \epsilon^{i+1}_{+}\Q^{H}_{X_{i+1}}[-1] \bigg)[-1].$$
Indeed,  we need to show that the induced map
$$\epsilon^{i-2}_{+}\Q^{H}_{X_{i-2}} \rightarrow \fC\bigg(\epsilon^{i}_{+}\Q^{H}_{X_{i}}, \epsilon^{i+1}_{+}\Q^{H}_{X_{i+1}} \bigg)[-1]  $$
is the zero map in $D^{b}MHM(X)$. Since the functor $rat: MHM(X) \rightarrow Perv(X)$ is faithful, and for any $\cM \in D^{b}MHM(X)$, there is an idenfication
$$rat(\cH^{i}(\cM)) = \hspace{.01in} ^{p}\cH^{i}(rat(\cM)),$$
it suffices to show the underlying map between the rational complexes
$$ \bR \epsilon^{i-2}_{*} \Q_{X_{i-2}} \rightarrow \fC\bigg( \bR \epsilon^{i}_{*}\Q_{X_{i}}, \bR \epsilon^{i+1}_{*}\Q_{X_{i+1}} \bigg)[-1]$$
is the zero map. Moreover, it suffices to show the map between the De Rham complexes
$$ \bR \epsilon^{i-2}_{*} \C_{X_{i-2}} \rightarrow \fC\bigg( \bR \epsilon^{i}_{*}\C_{X_{i}}, \bR \epsilon^{i+1}_{*}\C_{X_{i+1}} \bigg)[-1]$$
is the zero map. Recall from the previous section, the total complex $Tot(C)$ has a filtration such that 
$$Gr^{i}_{F}Tot(C) \simeq \bR \epsilon^{i}_{*}\C_{X_{i}}[-i].$$
The exact triangle
$$ \bT \fC\bigg( \bR \epsilon^{i}_{*}\C_{X_{i}}, \bR \epsilon^{i+1}_{*}\C_{X_{i+1}} \bigg)[-1] \arrow[r] & \bR \epsilon^{i}_{*}\C_{X_{i}} \ar[r] & \bR \epsilon^{i+1}_{*}\C_{X_{i+1}} \ar[r, "+1"] &  \hfill \eT$$
is identified with the exact sequence 
$$ \bT0 \ar[r] &  \bigg(F^{i}Tot(C)/F^{i+2}Tot(C) \bigg)[i] \arrow[r] & Gr^{i}_{F}Tot(C)[i] \ar[r] & Gr^{i+1}_{F}Tot(C)[i+1]  \ar[r] &  0 \eT$$
The map 
$$ \bR \epsilon^{i-2}_{*} \C_{X_{i-2}} \rightarrow \fC\bigg( \bR \epsilon^{i}_{*}\C_{X_{i}}, \bR \epsilon^{i+1}_{*}\C_{X_{i+1}} \bigg)[-1]$$
is identified with the composition of maps
$$\bT Gr^{i-2}_{F}Tot(C)[i-2] \arrow[rr, bend left, "zero-map"] \arrow[r] & Gr^{i-1}_{F}Tot(C)[i-1] \arrow[r] & \bigg(F^{i}Tot(C)/F^{i+2}Tot(C) \bigg)[i] \eT$$
 Therefore, the induced map
$$\epsilon^{i-2}_{+}\Q^{H}_{X_{i-2}} \rightarrow \fC\bigg(\epsilon^{i}_{+}\Q^{H}_{X_{i}}, \epsilon^{i+1}_{+}\Q^{H}_{X_{i+1}} \bigg)[-1]  $$
is the zero map. 

Through iteration, we have an object
$$\fC(\epsilon^{0}_{+}\Q^{H}_{X_{0}}, \fC(\epsilon^{1}_{+}\Q^{H}_{X_{1}}[-1],\fC(\cdots,\fC(\epsilon^{n-1}_{+}\Q^{H}_{X_{n-1}}[-(n-1)], \epsilon^{n}_{+}\Q^{H}_{X_{n}}[-(n-1)])))) \in D^{b}MHM(X).$$
\begin{thm}\label{Q-thm}
    Let $\epsilon: X_{\bullet} \rightarrow X$ be a cubical hyperresolution. The complex
    $$\fC(\epsilon^{0}_{+}\Q^{H}_{X_{0}}, \fC(\epsilon^{1}_{+}\Q^{H}_{X_{1}}[-1],\fC(\cdots,\fC(\epsilon^{n-1}_{+}\Q^{H}_{X_{n-1}}[-(n-1)], \epsilon^{n}_{+}\Q^{H}_{X_{n}}[-(n-1)])))) $$
    is isomorphic to $\Q^{H}_{X}[1]$  in $D^{b}MHM(X).$
\end{thm}

\begin{proof}
     If $a_{X}:X \rightarrow \{pt\}$ is the natural map, then $\Q^{H}_{X} = a^{*}_{X} \Q^{H}_{pt}$. Therefore,
    $$Hom(\Q^{H}_{X}[1], \fC(\epsilon^{0}_{+}\Q^{H}_{X_{0}}, \fC(\epsilon^{1}_{+}\Q^{H}_{X_{1}}[-1],\fC(\cdots,\fC(\epsilon^{n-1}_{+}\Q^{H}_{X_{n-1}}[-(n-1)], \epsilon^{n}_{+}\Q^{H}_{X_{n}}[-(n-1)])))))$$
    $$\cong Hom(a_{X}^{*}\Q^{H}_{pt}[1],  \fC(\epsilon^{0}_{+}\Q^{H}_{X_{0}}, \fC(\epsilon^{1}_{+}\Q^{H}_{X_{1}}[-1],\fC(\cdots,\fC(\epsilon^{n-1}_{+}\Q^{H}_{X_{n-1}}[-(n-1)], \epsilon^{n}_{+}\Q^{H}_{X_{n}}[-(n-1)])))))$$
    $$\cong Hom(\Q^{H}_{pt}[1], a_{X+} \fC(\epsilon^{0}_{+}\Q^{H}_{X_{0}}, \fC(\epsilon^{1}_{+}\Q^{H}_{X_{1}}[-1],\fC(\cdots,\fC(\epsilon^{n-1}_{+}\Q^{H}_{X_{n-1}}[-(n-1)], \epsilon^{n}_{+}\Q^{H}_{X_{n}}[-(n-1)])))))$$
    Now, 
    $$a_{X+} \fC(\epsilon^{0}_{+}\Q^{H}_{X_{0}}, \fC(\epsilon^{1}_{+}\Q^{H}_{X_{1}}[-1],\fC(\cdots,\fC(\epsilon^{n-1}_{+}\Q^{H}_{X_{n-1}}[-(n-1)], \epsilon^{n}_{+}\Q^{H}_{X_{n}}[-(n-1)]))))$$ 
    $$= \bR \Gamma(X, rat( \fC(\epsilon^{0}_{+}\Q^{H}_{X_{0}}, \fC(\epsilon^{1}_{+}\Q^{H}_{X_{1}}[-1],\fC(\cdots,\fC(\epsilon^{n-1}_{+}\Q^{H}_{X_{n-1}}[-(n-1)], \epsilon^{n}_{+}\Q^{H}_{X_{n}}[-(n-1)]))))))$$
    The functor $rat:D^{b}MHM(X) \rightarrow D^{b}(X)$ commutes with cones. Thus
    $$rat( \fC(\epsilon^{0}_{+}\Q^{H}_{X_{0}}, \fC(\epsilon^{1}_{+}\Q^{H}_{X_{1}}[-1],\fC(\cdots,\fC(\epsilon^{n-1}_{+}\Q^{H}_{X_{n-1}}[-(n-1)], \epsilon^{n}_{+}\Q^{H}_{X_{n}}[-(n-1)])))))$$
    $$\simeq  \fC(rat(\epsilon^{0}_{+}\Q^{H}_{X_{0}}), \fC(rat(\epsilon^{1}_{+}\Q^{H}_{X_{1}}[-1]),\fC(\cdots,\fC(rat(\epsilon^{n-1}_{+}\Q^{H}_{X_{n-1}}[-(n-1)]), rat(\epsilon^{n}_{+}\Q^{H}_{X_{n}}[-(n-1)])))))$$
    $$\simeq \fC(\bR \epsilon^{0}_{*}\Q_{X_{0}}, \fC(\bR \epsilon^{1}_{*}\Q_{X_{1}}[-1], \fC(\cdots,\fC(\bR \epsilon^{n-1}_{*}\Q_{X_{n-1}}[-(n-1)], \bR \epsilon^{n}_{*}\Q_{X_{n}}[-(n-1)]))))).$$
    The quasi-isomorphism above is compatible with the quasi-isomorphism \ref{ConeLemma}, but with $\Q$-coefficients. So we obtain the identifications
     $$Hom(\Q^{H}_{X}[1], \fC(\epsilon^{0}_{+}\Q^{H}_{X_{0}}, \fC(\epsilon^{1}_{+}\Q^{H}_{X_{1}}[-1],\fC(\cdots,\fC(\epsilon^{n-1}_{+}\Q^{H}_{X_{n-1}}[-(n-1)], \epsilon^{n}_{+}\Q^{H}_{X_{n}}[-(n-1)])))))$$
   
     $$= Hom(\Q^{H}_{pt}[1], \bR \Gamma(X, \Q_{X}[1])) = H^{0}(X, \Q_{X}) = \bigoplus \Q_{X_{\alpha}}.$$
     where $X_{\alpha}$ are the connected components of $X$. The elements $1_{\alpha} \in \Q_{X_{\alpha}}$ induce the natural map
     $$\Q^{H}_{X}[1] \rightarrow \fC(\epsilon^{0}_{+}\Q^{H}_{X_{0}}, \fC(\epsilon^{1}_{+}\Q^{H}_{X_{1}}[-1],\fC(\cdots,\fC(\epsilon^{n-1}_{+}\Q^{H}_{X_{n-1}}[-(n-1)], \epsilon^{n}_{+}\Q^{H}_{X_{n}}[-(n-1)]))))). $$
     Now, if $\mathcal{K}$ is the cone of the natural map above, then, by Lemma \ref{ConeLemma}, $rat(\mathcal{K}) \simeq 0,$ which implies $\mathcal{K} \simeq 0.$ So, the map 
     $$\Q^{H}_{X}[1] \rightarrow \fC(\epsilon^{0}_{+}\Q^{H}_{X_{0}}, \fC(\epsilon^{1}_{+}\Q^{H}_{X_{1}}[-1],\fC(\cdots,\fC(\epsilon^{n-1}_{+}\Q^{H}_{X_{n-1}}[-(n-1)], \epsilon^{n}_{+}\Q^{H}_{X_{n}}[-(n-1)]))))) $$
     is an isomorphism in $D^{b}MHM(X).$

\end{proof}
From the previous theorem, we obtain the following result, which was first shown by Saito \cite{saito5}.
\begin{cor}
    From the isomorphism 
    $$\Q^{H}_{X}[1] \rightarrow \fC(\epsilon^{0}_{+}\Q^{H}_{X_{0}}, \fC(\epsilon^{1}_{+}\Q^{H}_{X_{1}}[-1],\fC(\cdots,\fC(\epsilon^{n-1}_{+}\Q^{H}_{X_{n-1}}[-(n-1)], \epsilon^{n}_{+}\Q^{H}_{X_{n}}[-(n-1)]))))) $$
    there is a quasi-isomorphism
    $$Gr^{F}_{-p}DR(\Q^{H}_{X}) \simeq \Omp[-p]$$
\end{cor}

\begin{proof}
    If we apply $Gr^{F}_{-p}DR(\bullet)$ to the map
$$\Q^{H}_{X}[1] \rightarrow \fC(\epsilon^{0}_{+}\Q^{H}_{X_{0}}, \fC(\epsilon^{1}_{+}\Q^{H}_{X_{1}}[-1],\fC(\cdots,\fC(\epsilon^{n-1}_{+}\Q^{H}_{X_{n-1}}[-(n-1)], \epsilon^{n}_{+}\Q^{H}_{X_{n}}[-(n-1)]))))) $$
we obtain
$$Gr^{F}_{-p}DR(\Q^{H}_{X}[1])$$ 
$$\simeq Gr^{F}_{-p}DR \bigg[  \fC(\epsilon^{0}_{+}\Q^{H}_{X_{0}}, \fC(\epsilon^{1}_{+}\Q^{H}_{X_{1}}[-1],\fC(\cdots,\fC(\epsilon^{n-1}_{+}\Q^{H}_{X_{n-1}}[-(n-1)], \epsilon^{n}_{+}\Q^{H}_{X_{n}}[-(n-1)]))))) \bigg]$$
$$\simeq  \fC(Gr^{F}_{-p}DR(\epsilon^{0}_{+}\Q^{H}_{X_{0}}),\fC(\cdots,\fC(Gr^{F}_{-p}DR(\epsilon^{n-1}_{+}\Q^{H}_{X_{n-1}})[-(n-1)], Gr^{F}_{-p}DR(\epsilon^{n}_{+}\Q^{H}_{X_{n}})[-(n-1)])))))$$
$$\fC(\bR \epsilon^{0}_{*}\Omega^{p}_{X_{0}}[-p], \fC(\bR \epsilon^{1}_{*}(\Omega^{p}_{X_{1}}[-p])[-1],\fC(\cdots, \fC(\bR \epsilon^{n-1}_{*}(\Omega^{p}_{X_{n-1}}[-p])[-(n-1)], \bR \epsilon^{n}_{*}(\Omega^{p}_{X_{n}}[-p])[-(n-1)])))))$$
$$\simeq (\Omp[-p])[1].$$
\end{proof}

Let $\mathcal{U} \subseteq X$ be any open subset with inclusion map $\mathbf{j}:\cU \hookrightarrow X.$ For each map $\epsilon^{i}: X_{i} \rightarrow X$, we have a cartesian diagram
$$\bT (\epsilon^{i})^{-1}(\cU)\arrow[r, "\epsilon^{i}"] \arrow[d, "\bfj^{i}"] & \cU \arrow[d, "\bfj"] \\
X_{i} \arrow[r, "\epsilon^{i}"] & X \eT$$
Set $\cU_{i} = (\epsilon^{i})^{-1}(\cU)$ with the inclusion map $\bfj^{i}: \cU_{i} \hookrightarrow X_{i}$. Using that $\epsilon^{i}$ is proper,
$$\epsilon^{i}_{+}\Q^{H}_{X_{i}} \otimes \bfj_{+}\Q^{H}_{\cU} \simeq \bfj_{+}\bigg(\epsilon^{i}_{+}\Q^{H}_{\cU_{i}} \otimes \Q^{H}_{\cU} \bigg) \simeq \bfj_{+}\epsilon^{i}_{+}\Q^{H}_{\cU_{i}} \simeq \epsilon^{i}_{+}\bfj^{i}_{+}\Q^{H}_{\cU_{i}}    $$
We set $\Sigma_{i} = (X_{i} \backslash \cU_{i})_{red}$.  We may assume $\Sigma_{i}$ on each irreducible component of $X_{i}$ is either empty, equal to the irreducible component, or a simple normal crossing divisor on the irreducible component of $X_{i}$. From Theorem \ref{Q-thm},
$$\bfj_{+} \Q^{H}_{\cU}[1] \simeq \Q^{H}_{X}[1] \otimes \bfj_{+} \Q^{H}_{\cU} $$
$$\simeq \fC(\epsilon^{0}_{+}\Q^{H}_{X_{0}}, \fC(\epsilon^{1}_{+}\Q^{H}_{X_{1}}[-1],\fC(\cdots,\fC(\epsilon^{n-1}_{+}\Q^{H}_{X_{n-1}}[-(n-1)], \epsilon^{n}_{+}\Q^{H}_{X_{n}}[-(n-1)]))))) \otimes \bfj_{+} \Q^{H}_{\cU}$$
$$\simeq  \fC(\epsilon^{0}_{+}\Q^{H}_{X_{0}} \otimes \bfj_{+} \Q^{H}_{\cU}, \fC(\epsilon^{1}_{+}\Q^{H}_{X_{1}}\otimes \bfj_{+} \Q^{H}_{\cU}[-1],\fC(\cdots,\fC(\epsilon^{n-1}_{+}\Q^{H}_{X_{n-1}}\otimes \bfj_{+} \Q^{H}_{\cU}[-(n-1)], \epsilon^{n}_{+}\Q^{H}_{X_{n}}\otimes \bfj_{+} \Q^{H}_{\cU}[-(n-1)])))))$$
$$\simeq  \fC(\epsilon^{0}_{+}\bfj^{0}_{+} \Q^{H}_{\cU_{0}}, \fC(\epsilon^{1}_{+} \bfj^{1}_{+} \Q^{H}_{\cU_{1}}[-1],\fC(\cdots,\fC(\epsilon^{n-1}_{+} \bfj^{n-1}_{+} \Q^{H}_{\cU_{n-1}}[-(n-1)], \epsilon^{n}_{+} \bfj^{n}_{+} \Q^{H}_{\cU_{n}}[-(n-1)]))))).$$ 
If  $\Sigma = (X \backslash \cU)_{red}$, by \cite[\S 3]{saito5},  the $p^{th}$-graded piece of the Du Bois complex of the pair $(X, \Sigma)$ is given by 
$$\underline{\Omega}^{p}_{X}(\log \Sigma)\simeq Gr^{F}_{-p}DR(\bfj_{+}\Q^{H}_{\cU})[p]$$
$$\simeq Gr^{F}_{-p}DR\bigg(\fC(\epsilon^{0}_{+}\bfj^{0}_{+} \Q^{H}_{\cU_{0}}, \fC(\epsilon^{1}_{+} \bfj^{1}_{+} \Q^{H}_{\cU_{1}}[-1],\fC(\cdots,\fC(\epsilon^{n-1}_{+} \bfj^{n-1}_{+} \Q^{H}_{\cU_{n-1}}[-(n-1)], \epsilon^{n}_{+} \bfj^{n}_{+} \Q^{H}_{\cU_{n}}[-(n-1)]))))\bigg)[p-1]$$
$$\simeq \fC(\bR \epsilon^{0}_{*} \Omega^{p}_{X_{0}}(\log \Sigma_{0}), \fC(\bR \epsilon^{1}_{*}\bR  \Omega^{p}_{X_{1}}(\log \Sigma_{1})[-1],\fC(\cdots)))[-1].$$
With a similar calculation, 
$$\underline{\Omega}^{p}_{X, \Sigma}\simeq Gr^{F}_{-p}DR(\bfj_{!}\Q^{H}_{\cU})[p]$$
$$\simeq Gr^{F}_{-p}DR\bigg(\fC(\epsilon^{0}_{+}\bfj^{0}_{!} \Q^{H}_{\cU_{0}}, \fC(\epsilon^{1}_{+} \bfj^{1}_{!} \Q^{H}_{\cU_{1}}[-1],\fC(\cdots,\fC(\epsilon^{n-1}_{+} \bfj^{n-1}_{!} \Q^{H}_{\cU_{n-1}}[-(n-1)], \epsilon^{n}_{+} \bfj^{n}_{!} \Q^{H}_{\cU_{n}}[-(n-1)]))))\bigg)[p-1]$$
$$\simeq \fC(\bR \epsilon^{0}_{*} \Omega^{p}_{X_{0}}(\log \Sigma_{0})(-\Sigma_{0}), \fC(\bR \epsilon^{1}_{*}\Omega^{p}_{X_{1}}(\log \Sigma_{1})(-\Sigma_{1})[-1],\fC(\cdots)))[-1].$$

\subsection{Simplicial Resolutions And Cyclic Coverings}\label{singularCyclic}
Recall, we have arbitrary line bundle  $\cL$ on $X$ such that for a positive integer $N \geq 0$, there exists a global section $s \in \Gamma(X, \cL^{N})$ that defines an effective Cartier divisor $D.$ We set $U = X \backslash D$ with natural map $j:U \hookrightarrow X.$ If $\pi: Y \rightarrow X$ is the $N$-fold cyclic covering of $D$, there exists $\cV \in D^{b}MHM(X)$ such that
    $$\pi_{+}\Q^{H}_{Y} \simeq \Q^{H}_{X} \oplus j_{!}\cV.$$
    \begin{rmk}
        We are excluding the shift of $\Q^{H}_{Y}$ by the dimension $\dim Y = n$, which will be more convenient for working with simplicial resolutions. 
    \end{rmk}
    For $0 \leq p \leq n$, applying the functor $Gr^{F}_{-p}DR(\bullet)$ to the isomorphism $ \pi_{+}\Q^{H}_{Y} \simeq \Q^{H}_{X} \oplus \cV$, we have
    $$ \OmpY[-p] \simeq \Omp[-p] \oplus Gr^{F}_{-p}DR(j_{!}\cV).$$
    From our previous discussion, we can describe the complex $Gr^{F}_{-p}DR(\cV)$ in terms of cubical hyperresolution. Let $\epsilon: X_{\bullet} \rightarrow X$ be a cubical hyperresolution with the proper maps
$$\epsilon_{ij}:X_{i} \rightarrow X_{i-1} \quad \text{for $0 \leq j \leq i$.}$$
$$\epsilon^{i}:X_{i} \rightarrow X.$$
For each $i$, there is a commutative diagram
$$ \bT Y_{i} \arrow[d, "\pi_{i}"] \arrow[r,"\varepsilon^{i}"]& Y  \ar[d, "\pi"] \\
  X_{i} \arrow[r, "\epsilon^{i}"] & X \eT$$
  where $Y_{i} = (Y \times_{X} X_{i})_{red}$ is the $N$-fold cyclic covering of  $D_{i}=(\epsilon^{i})^{*}D$. Again, we may assume that $D_{i}$ on each irreducible component of $X_{i}$ is either empty, the irreducible component, or  $E_{i}=(D_{i})_{red}$ is a simple normal crossing divisor on the irreducible component. There is also the commutative diagram 
  $$ \bT  \phi^{-1}_{i}(U)=U_{i} \arrow[d, "j_{i}"] \arrow[r, "\epsilon^{i}"]& U  \ar[d, "j"] \\
  X_{i} \arrow[r, "\epsilon^{i}"] & X \eT$$
  From proper base change,
  $$\pi_{i+}\Q^{H}_{Y_{i}} \simeq \pi_{i+}(\varepsilon^{i})^{*}\Q^{H}_{Y} \simeq (\epsilon^{i})^{*}\bigg(\Q^{H}_{X}\oplus j_{!}\cV\bigg) \simeq \Q^{H}_{X_{i}} \oplus j_{i !}\cV_{i},$$
  where $\cV_{i}:= (\epsilon^{i})^{*}\cV.$ 
  We recently proved
  $$\Q^{H}_{X} \simeq \fC(\epsilon^{0}_{+}\Q^{H}_{X_{0}}, \fC(\epsilon^{1}_{+}\Q^{H}_{X_{1}}[-1],\fC(\cdots,\fC(\epsilon^{n-1}_{+}\Q^{H}_{X_{n-1}}[-(n-1)], \epsilon^{n}_{+}\Q^{H}_{X_{n}}[-(n-1)])))))[-1].$$
  Therefore, 
  $$j_{!}\cV \simeq  \Q^{H}_{X} \otimes j_{!}\cV $$
  $$\simeq \fC(\epsilon^{0}_{+}\Q^{H}_{X_{0}}, \fC(\epsilon^{1}_{+}\Q^{H}_{X_{1}}[-1],\fC(\cdots,\fC(\epsilon^{n-1}_{+}\Q^{H}_{X_{n-1}}[-(n-1)], \epsilon^{n}_{+}\Q^{H}_{X_{n}}[-(n-1)])))))[-1] \otimes j_{!}\cV$$
  $$\simeq \fC(\epsilon^{0}_{+}\Q^{H}_{X_{0}}\otimes j_{!}\cV, \fC(\epsilon^{1}_{+}\Q^{H}_{X_{1}}\otimes j_{!}\cV[-1],\fC(\cdots,\fC(\epsilon^{n-1}_{+}\Q^{H}_{X_{n-1}}\otimes j_{!}\cV[-(n-1)], \epsilon^{n}_{+}\Q^{H}_{X_{n}}\otimes j_{!}\cV[-(n-1)])))))[-1]$$
  From proper base change and the projection formula,
      $$\epsilon^{\bullet}_{+}\Q^{H}_{X_{\bullet}} \otimes j_{!}\cV \simeq j_{!}\bigg( j^{*}\epsilon^{\bullet}_{+}\Q^{H}_{X_{\bullet}} \otimes \cV \bigg) \simeq j_{!}\bigg( \epsilon^{\bullet}_{+}\Q^{H}_{U_{\bullet}} \otimes \cV \bigg) \simeq j_{!}\bigg( \epsilon^{\bullet}_{+}(\Q^{H}_{U_{\bullet}} \otimes (\epsilon^{\bullet})^{*}\cV) \bigg) \simeq  j_{!}\bigg(\epsilon^{\bullet}_{+}\cV_{\bullet}\bigg) \simeq \epsilon^{\bullet}_{+}j_{\bullet!}\cV_{\bullet}$$
        Thus, 
        $$j_{!}\cV \simeq  \fC( \epsilon^{0}_{+}j_{0!}\cV_{0}, \fC( \epsilon^{1}_{+}j_{1!}\cV_{1}[-1],\fC(\cdots,\fC( \epsilon^{n-1}_{+}j_{(n-1)!}\cV_{n-1}[-(n-1)],  \epsilon^{n}_{+}j_{n!}\cV_{n}[-(n-1)])))))[-1].$$
  Since $X_{\bullet}$ is smooth, there are quasi-isomorphisms
     $$Gr^{F}_{-p}DR\bigg((j_{\bullet})_{!}\cV_{\bullet}\bigg) \simeq \bigg(\bigoplus_{i=1}^{N-1}\Omega^{p}_{X_{\bullet}}\log(E_{\bullet})\otimes \cO_{X_{\bullet}}(D^{>0}_{\bullet i}))\otimes \cL_{\bullet}^{-i}\bigg)[-p]$$
     $$Gr^{F}_{-p}DR\bigg(\epsilon^{\bullet}_{+}(j_{\bullet})_{!}\cV_{\bullet}\bigg) \simeq \bR \epsilon^{\bullet}_{*}\bigg(\bigoplus_{i=1}^{N-1}\Omega^{p}_{X_{\bullet}}\log(E_{\bullet})\otimes \cO_{X_{\bullet}}(D^{>0}_{\bullet i}))\otimes \cL_{\bullet}^{-i}\bigg)[-p]$$
    $$ \simeq\bigg(\bigoplus_{i=1}^{N-1} \bR \epsilon^{\bullet}_{*}(\Omega^{p}_{X_{\bullet}}\log(E_{\bullet})\otimes \cO_{X_{\bullet}}(D^{>0}_{\bullet i}))\bigg)\otimes \cL^{-i}[-p].$$
    As filtered $D_{U_{\bullet}}$-modules, there is a splitting $\cV_{\bullet} \cong \displaystyle \bigoplus_{i=1}^{N-1}\cH_{\bullet i}$ and there are quasi-isomorphisms
$$Gr^{F}_{-p}DR(j_{\bullet !}\cH_{\bullet i})[p-n] \simeq(\Omega^{p}_{X_{\bullet}}\log(E_{\bullet})\otimes \cO_{X_{\bullet}}(D^{>0}_{\bullet i}))\otimes \cL_{\bullet}^{-i}[-p] \quad \text{for $1 \leq i \leq N-1.$}$$
      With the identifications above, we obtain the following quasi-isomorphisms,
     $$ \pi_{*} \OmpY \simeq  \Omp \oplus Gr^{F}_{-p}DR(j_{!}\cV)[p]$$
     $$\simeq \Omp \oplus Gr^{F}_{-p}DR\bigg( \fC( \epsilon^{0}_{+}j_{0!}\cV_{0}, \fC( \epsilon^{1}_{+}j_{1!}\cV_{1}[-1],\fC(\cdots,\fC( \epsilon^{n-1}_{+}j_{(n-1)!}\cV_{n-1}[-(n-1)],  \epsilon^{n}_{+}j_{n!}\cV_{n}[-(n-1)])))))[-1]\bigg)[p]$$
      $$\simeq\Omp \oplus \bigoplus_{i=1}^{N-1} \fC(\bR \epsilon^{0}_{*}(\Omega^{p}_{X_{0}}\log(E_{0}) \otimes \cO_{X_{0}}(D^{>0}_{0i}), \fC(\bR \epsilon^{1}_{*}(\Omega^{p}_{X_{1}}\log(E_{1}) \otimes \cO_{X_{1}}(D^{>0}_{1i})[-1]),\fC(\cdots))[-1] \otimes \cL^{-i}$$
      For $1 \leq i \leq N-1$, the complexes
      $$ \fC(\bR \epsilon^{0}_{*}(\Omega^{p}_{X_{0}}\log(E_{0}) \otimes \cO_{X_{0}}(D^{>0}_{0i}), \fC(\bR \epsilon^{1}_{*}(\Omega^{p}_{X_{1}}\log(E_{1}) \otimes \cO_{X_{1}}(D^{>0}_{1i})[-1]),\fC(\cdots))[-1]$$
      are independent of the simplicial resolution (up to quasi-isomorphism). 
\begin{defn}
  Take an arbitrary line bundle  $\cL$ on $X$. Assume there is a positive integer $N$ such that there exists a global section $s \in \Gamma(X, \cL^{N})$ that defines an effective Cartier divisor $D.$  Let $\epsilon: X_{\bullet} \rightarrow X$ be a cubical hyperresolution of the pair $(X,D)$. Then $\underline{\Omega}^{p}_{X}( \log D ) (D_{i}^{>0} )$ is defined to be (up to quasi-isomorphism) the complex
      $$\fC(\bR \epsilon^{0}_{*}(\Omega^{p}_{X_{0}}\log(E_{0}) \otimes \cO_{X_{0}}(D^{>0}_{0i}), \fC(\bR \epsilon^{1}_{*}(\Omega^{p}_{X_{1}}\log(E_{1}) \otimes \cO_{X_{1}}(D^{>0}_{1i})[-1]),\fC(\cdots))[-1].$$
\end{defn}

       \begin{thm}
       Let $\cL$ be an arbitrary line bundle on an algebraic variety $X$ of dimension $n$. Suppose that for a positive integer $N$ there exists a global section $s \in \Gamma(X, \cL^{N})$
    that defines an effective Cartier divisor $D.$ Let $\pi: Y \rightarrow X$ denote the $N$-fold cyclic covering of $D$. For $0 \leq p \leq n$, 
    $$\pi_{*}\underline{\Omega}^{p}_{Y} \simeq \Omp \oplus \bigoplus_{i=1}^{N-1}(\underline{\Omega}^{p}_{X}(\log D) (D_{i}^{>0})) \otimes \cL^{-i}.$$
  \end{thm}

      \begin{rmk}
         Let $\epsilon: X_{\bullet} \rightarrow X$ be cubical hyperresolution. If we choose $D$ sufficiently general, for example, when $D$ is a sufficiently general section of a basepoint-free linear system, then  $D_{\alpha} = D \times X_\alpha$ is a smooth divisor on $X_{\alpha}$, $D_{\alpha} \neq 0$ on each irreducible component, and 
         $$D_{\bullet}:= D \times X_{\bullet} \rightarrow D$$
         is a cubical hyperresolution. The variety $Y_{\alpha}$ from the cyclic covering $Y_{\alpha} \rightarrow X_{\alpha}$ is smooth, and the divisors $D^{>0}_{\alpha i}$ are all trivial. Thus, for $D$ sufficiently general,
        $$\pi_{*}\underline{\Omega}^{p}_{Y} \simeq \Omp \oplus \bigoplus_{i=1}^{N-1}\underline{\Omega}^{p}_{X}(\log D)\otimes \cL^{-i},$$
        where $\underline{\Omega}^{p}_{X}(\log D)$ is the $p^{th}$-graded piece of the Du Bois complex of the pair $(X, D)$ \cite{dubois}. This quasi-isomorphism was also recently shown by Kov\'acs  \cite[Thm. 6.2.(iv)]{kovacsINJ}.

     \end{rmk}

     \begin{prop}\label{DB-prop}
     Let $\pi:Y \rightarrow X$ be the $N$-fold cyclic covering of $D$ and assume $X$ has Du Bois singularities. Let $\epsilon:X_{\bullet} \rightarrow X$ be a cubical hyperresolution of the pair $(X, D)$, and $\pi_{i}:Y_{i} \rightarrow X_{i}$ the $N$-fold cyclic covering of $ (\epsilon^{i})^{*}D = D_{i}$. For each $i,$ assume $D_{i} \neq 0$ on each irreducible component of $X_{i}$. If $Y_{i}$ has Du Bois singularities for each $i$, then $Y$ has Du Bois singularities.
     \end{prop}

     \begin{proof}
         If $Y_{\bullet}$ has Du Bois singularities, then 
         $$(\pi_{\bullet})_{*}Gr^{F}_{0}DR(\Q^{H}_{Y_{\bullet}})\simeq \bigoplus_{i=0}^{N-1}\cO_{X_{\bullet}}\otimes\epsilon^{*}_{\bullet}\cL^{-i}.$$
          $$\bR\epsilon^{\bullet}_{*}(\pi_{\bullet})_{*}Gr^{F}_{0}DR(\Q^{H}_{Y_{\bullet}})\simeq \bigoplus_{i=0}^{N-1}\bR \epsilon^{\bullet}_{*}(\cO_{X_{\bullet}} \otimes \epsilon^{*}_{\bullet}\cL^{-i}) \simeq \bigoplus_{i=0}^{N-1}\bR \epsilon^{\bullet}_{*}(\cO_{X_{\bullet}}) \otimes \cL^{-i}.$$
         Thus
  $$\pi_{*}\underline{\Omega}^{0}_{Y} \simeq \Om \oplus \bigoplus_{i=1}^{N-1}(\Om ( \log D)(D_{i}^{>0})) \otimes \cL^{-i}  \simeq \Om \oplus \bigoplus_{i=1}^{N-1}\Om\otimes \cL^{-i}$$
         
         $$\bigoplus_{i=0}^{N-1}\Om\otimes \cL^{-i}\simeq \bigoplus_{i=0}^{N-1} \cL^{-i} \cong \pi_{*}\cO_{Y}$$
         Since $\pi: Y \rightarrow X$ is affine, we can conclude that $Y$ has Du Bois singularities.
          \end{proof}

           \begin{prop}
     Let $\pi:Y \rightarrow X$ be the $N$-fold cyclic covering of $D$ and assume $X$ has rational singularities. Let $\epsilon:X_{\bullet} \rightarrow X$ be cubical hyperresolution of the pair $(X, D)$ and $X^{rs}_{0} \subseteq X_{0}$ the subvariety such that $X^{res}_{0} \rightarrow X$ is a resolution of singularities. Let
     $$\pi_{0}:Y^{res}_{0} \rightarrow X^{res}_{0}$$
     the $N$-fold cyclic covering of $(\epsilon^{0})^{*}D \cap X^{res}_{0} = D_{0} \cap X^{res}_{0}$. If $Y^{res}_{0}$ has rational singularities, then $Y$ has rational singularities.
     \end{prop}

     \begin{proof}
         Since $X$ is normal and the problem is local, it suffices to consider the case where $X$ is irreducible. Also, because we are considering the top differential,
         $$\underline{\Omega}^{n}_{X}(\log D)(D_{i}^{>0}) = \underline{\Omega}^{n}_{X}(\log D)(D_{i}^{>0})\bigg \vert_{X_{0}^{res}}$$
         So, we may assume $X_{0} = X^{res}_{0}.$ If $g:\tY \rightarrow Y$ is a resolution of singularities, then
         $$\bR g_{*}\cO_{\tY} \simeq \cD(\underline{\Omega}^{n}_{Y})$$
         $$\pi_{*}\bR g_{*}\cO_{\tY} \simeq \pi_{*}\cD(\underline{\Omega}^{n}_{Y}) \simeq \cD(\underline{\Omega}^{n}_{X}) \oplus \bigoplus_{i=1}^{N-1}\cD(\underline{\Omega}^{n}_{X}(\log D) (D_{i}^{>0}) \otimes \cL^{-i}).$$
         For $1 \leq i \leq N-1$, there are quasi-isomorphism
         $$ \cD(\underline{\Omega}^{n}_{X}(\log D)(D_{i}^{>0}) \otimes \cL^{-i})= \cD(\bR\epsilon^{0}_{*}(\Omega^{n}_{X_{0}}(E_{0}) \otimes \cO_{X_{0}}(D^{>0}_{0i}))\otimes \cL^{-i})\simeq \bR \epsilon^{0}_{*}( \cO_{X_{0}}(D^{\geq 0}_{0i}))\otimes \cL^{-i}.$$
      If $Y_{0}$ and $X$ have rational singularities, then by Corollary \ref{singCor},
    $$\bR\epsilon^0_{*}( \cO_{X_{0}}(D^{\geq 0}_{0i})) \simeq \bR\epsilon^0_{*}\cO_{X_{0}} \simeq \cO_{X} \quad \text{for $1 \leq i \leq N-1.$}$$
 Thus
    $$ \pi_{*}\bR g_{*}\cO_{\tY} \simeq \pi_{*}\cD(\underline{\Omega}^{n}_{Y}) \simeq \bigoplus_{i=0}^{N-1}\cL^{-i} \cong \pi_{*}\cO_{Y}.$$
    Since $\pi:Y \rightarrow X$ is affine, $Y$ has rational singularities.

     \end{proof}

             \begin{ex}
    In Proposition \ref{DB-prop}, it was assumed $D_{i} \neq 0$ for each irreducible component of $X_{i}$. We will give an example of why this assumption is needed. Consider the $2$-fold cyclic covering 
    $$\pi: Y = \Spec \bigg( \C[x, y, t]/\langle y^{2}-x^2, t^2-xy\rangle \bigg)\rightarrow \Spec\bigg(\C[x,y]/\langle  y^2 -x^2 \rangle\bigg)=X$$
    with $D = \{xy=0\}$ on $X$. If $\phi:\tX \rightarrow X$ is the blow-up of $p=(0,0)$, then $\tX$ is a resolution of singularities. Let $\{q_{1}, q_{2}\}=E \subset \tX$ denote the divisor. Then
    $$\bT E \arrow[r, shift left] \arrow[r, shift right] & \tX \coprod p \arrow[r] & X \eT$$
    is a simplicial resolution of $X.$ Because the pull-back of $D$ to $E$ and $p$ is all of $E$ and $p$, respectively, we have the identifications
    $$\pi_{*}\underline{\Omega}^{0}_{Y} \simeq \Om \oplus Gr^{F}_{0}DR(j_{!}\cV) \simeq \cO_{X} \oplus Gr^{F}_{0}DR(j_{!}\cV)$$
    $$Gr^{F}_{0}DR(j_{!}\cV) \simeq Gr^{F}_{0}DR\bigg[\fC (\epsilon^{0}_{+}(\epsilon^{0})^{*}(j_{!}\cV), \epsilon^{1}_{+}(\epsilon^{1})^{*}(j_{!}\cV))\bigg][-1] \simeq Gr^{F}_{0}DR \bigg[\fC(\phi_{+}\phi^{*}(j_!\cV), 0) \bigg][-1]$$
    $$\simeq  \bR\phi_{*}Gr^{F}_{0}DR((\phi^{*})(j_!\cV))$$
   The logarithmic connection
    $$\phi^{*}\cL^{-1} \rightarrow \Omega^{1}_{\tX}(\log E) \otimes \phi^{*}\cL^{-1}$$
    has the eigenvalues of $res \nabla$ along the irreducible components of $E$ to be contained in $(0,1]$. So, 
    $$\phi_{*}Gr^{F}_{0}DR((\phi^{*})(j_!\cV)) \simeq \phi_{*}\phi^{*}\cL^{-1} \cong \phi_{*}\cO_{\tX} \otimes \cL^{-1}$$
    Thus, $Y$ does not have Du Bois singularities,
     $$\pi_{*}\cH^{0}(\underline{\Omega}^{0}_{Y}) \simeq \cO_{X} \oplus(\phi_{*}\cO_{\tX} \otimes \cL^{-1}) \not \cong \cO_{X} \oplus \cL^{-1} \cong \pi_{*}\cO_{Y}.$$
    
         \end{ex}

       \section{Vanishing Theorems}\label{S5}
       For the last section of this paper,  $X$ will be a projective variety, possibly singular, of dimension $n$.  Again, we are in the setting of \ref{setting} and $\pi:Y \rightarrow X$ is the $N$-fold cyclic covering of $D$. If $j: U= X \backslash D \hookrightarrow X$ is the natural map, then from Section \ref{S2} there exists $\cV \in D^{b}MHM(U)$  such that we have the following isomorphism in the derived category of mixed Hodge modules,
    $$\pi_{+}\Q^{H}_{Y}[n] \simeq \Q^{H}_{X}[n] \oplus j_{!}\cV.$$
    Thus obtaining isomorphisms
    $$ H^{i}(Y, \Q) \cong H^{i}(X, \Q) \oplus H^{i-n}_{c}(U, rat(\cV)) \quad \text{for $i \in \Z$.}$$ 
    Since the morphism $\pi:Y \rightarrow X$ is finite, the functor $\pi_{+}:D^{b}MHM(Y) \rightarrow D^{b}MHM(X)$ is exact. So, we must have $rat(\cV) \in \hspace{.01in} ^{p}D^{\leq 0}(U)$ because $\Q_{Y}[n] \in \hspace{.01in}^{p}D^{\leq 0}(Y)$. Moreover, because the map $\pi:Y \backslash \pi^{*}D \rightarrow X \backslash D$ is \'etale, if we set
    $$\lcdef(U) = \max\{\ell \in \N_{0}| \hspace{.01in}^{p} \cH^{-\ell}(\Q_{U}[n])\neq 0\}$$
    which is called \emph{the local cohomological defect of $U$}, then $^{p}\cH^{i}(rat(\cV)) = 0$ for $i< -\lcdef(U).$ For more information on the local cohomological defect, see \cite{PopaShen}. So, if $U$ is affine, then by an application of Artin's vanishing theorem for perverse sheaves, see \cite[Thm. 4.1.1]{BBD}, we have
    $$ H^{i-n}_{c}(U, rat(\cV)) = 0 \quad \text{for $i-n<-\lcdef(U)$.}$$ 
    The $\Q$-vector space $H^{i-n}_{c}(U, rat(\cV))$ has a mixed Hodge structure by \cite{saito2}. From our discussion in Section \ref{S4}, there is a decomposition
    $$ H^{i-n}_{c}(U, rat(\cV)) \otimes \C \cong \bigoplus _{p \in \Z}H^{i-p}(X, Gr^{F}_{-p}DR( j_{!}\cV))$$
    $$\cong \bigoplus _{p \in \Z} \bigoplus_{1 \leq i \leq N-1}H^{i-p}(X,\underline{\Omega}^{p}_{X}(\log D) (D_{i}^{>0})) \otimes \cL^{-i}).$$
    So, if $i<  n - \lcdef(U)$, then 
    $$H^{i-p}(X,\underline{\Omega}^{p}_{X}(\log D) (D_{1}^{>0})) \otimes \cL^{-1}) = 0.$$
    
\begin{thm}
       Let $\cL$ be an arbitrary line bundle on a projective variety $X$ of dimension $n$. Suppose that for a positive integer $N$ there exists a global section $s \in \Gamma(X, \cL^{N})$
    that defines an effective Cartier divisor $D$ such that $U = X \backslash D$ is an affine variety. If $\lcdef(U)$ is the local cohomological defect of $U$, then
    $$ H^{q}(X,\underline{\Omega}^{p}_{X}(\log D) (D_{1}^{>0})) \otimes \cL^{-1}) = 0 \quad \text{for $p + q < n - \lcdef(U)$.}$$
\end{thm}

\begin{thm}\label{nefthm}
    Let $X$ be an irreducible projective variety of dimension $n$, and $\cL$ a big and nef line bundle on $X$.  If $\lcdef(X)$ is the local cohomolgical defect of $X$, then
    $$ H^{i}(X,\underline{\Omega}^{0}_{X} \otimes \cL^{-1}) = 0 \quad \text{for $ i< n - \lcdef(X)$.}$$
\end{thm}

\begin{proof}
 Since $\cL$ is big and nef,  there exists an effective divisor $D$ such that for $N \gg 0$, there is an isomorphism $\cL^{N} \otimes \cO_{X}(-D) \cong \cO_{X}(\cA_{N})$, where $\cA_{N}$ is an ample divisor. Furthermore, we may assume the linear system $|\cA_{N}|$ is basepoint-free \cite[Thm. 2.3.9]{LarasfeldII}.  Let $H_{N}$ be a general member of the linear system $|\cA_{N}|$. Then the open set $U_{N} = X \backslash (H_N +D)$ is affine. By the previous theorem 
    $$ H^{i}(X,\underline{\Omega}^{0}_{X} ((D+H_{N})_{1}^{>0})) \otimes \cL^{-1}) = 0 \quad \text{for $ i< n - \lcdef(U_{N})$.}$$
    Since $U_{N}$ is an open subset of $X$, we have $\lcdef(U_{N}) \leq \lcdef(X)$. Therefore, 
     $$ H^{i}(X,\underline{\Omega}^{0}_{X} (D+H_{N})_{1}^{>0} \otimes \cL^{-1}) = 0 \quad \text{for $ i< n - \lcdef(X)$.}$$
     Let $\epsilon: X_{\bullet} \rightarrow X$ be a cubical hyperresolution of the pair $(X, D)$. As $H_{N}$ is a general member of a basepoint-free linear system, $\epsilon_{H_{N}}:H_{\bullet(N)}=H_{N} \times _{X}X_{\bullet} \rightarrow H_{N}$ is also a hyperresolution. Furthermore, $H_{k(N)} + D_{k}$ has simple normal crossing support on $X_{k}$ for each $k$ such that $D_{k} \neq X_{k}.$ Recall,
      $$\underline{\Omega}^{0}_{X} (D +H_{N})_{1}^{>0} )\simeq$$
      $$\fC(\bR \epsilon^{0}_{*}(\cO_{X_{0}}((H_{0(N)}+D_{0})^{>0}_{1}), \fC(\bR \epsilon^{1}_{*}\ \cO_{X_{1}}(H_{1(N)}+D_{1})^{>0}_{1})[-1]),\fC(\cdots))[-1].$$
     Now choose $N\gg 0$ so that $\cO_{X_k}(H_{k(N)}+D_{k})^{>0}_{1}) = \cO_{X_{k}}$ for all $k$, which gives us $\underline{\Omega}^{0}_{X} (D+H_{N})_{1}^{>0} = \underline{\Omega}^{0}_{X}$ for $N \gg 0.$
\end{proof}

\begin{cor}
     Let $X$ be an irreducible projective variety of dimension $n$, and $\cL$ a big and nef line bundle on $X$. If the local cohomological defect of $X$ is zero (e.g. $X$ is a local complete intersection), then
    $$ H^{i}(X, \Om \otimes \cL^{-1}) = 0 \quad \text{for $ i< n.$}$$
\end{cor}

\printbibliography

\end{document}